\documentclass[11pt]{amsart}
\usepackage{amsmath,amsthm,amscd,amsfonts, amssymb, mathrsfs}
\usepackage{graphicx}
\usepackage{tikz}
\usepackage{color}
\include{diagxy}
\usepgflibrary{shapes}
\newcommand{\sect}[1]{\section{#1}\setcounter{equation}{0}}

\font\mbn=msbm10 scaled \magstep1
\font\mbs=msbm7 scaled \magstep1
\font\mbss=msbm5 scaled \magstep1
\newfam\mbff
\textfont\mbff=\mbn
\scriptfont\mbff=\mbs
\scriptscriptfont\mbff=\mbss

\newtheorem{Th}{Theorem}[section]
\newtheorem{Lm}[Th]{Lemma}
\newtheorem{C}[Th]{Corollary}
\newtheorem{D}[Th]{Definition}
\newtheorem{Proposition}[Th]{Proposition}
\newtheorem{R}[Th]{Remark}
\newtheorem{Problem}[Th]{Problem}
\newtheorem{E}[Th]{Example}

\begin{document}
\title[On Bernstein Classes of Quasianalytic Maps]{On Bernstein Classes of Quasianalytic Maps}

\author{Alexander Brudnyi} 
\address{Department of Mathematics and Statistics\newline
\hspace*{1em} University of Calgary\newline
\hspace*{1em} Calgary, Alberta\newline
\hspace*{1em} T2N 1N4}
\email{albru@math.ucalgary.ca}

\thanks{Research supported in part by NSERC}

\begin{abstract}
We study the structure of families of ``well approximable'' elements of tensor products of Banach spaces including analogs of the classical quasianalytic classes in the sense of Bernstein and Beurling. As in the case of quasianalytic functions, we prove for members of these families variants of
the Mazurkiewicz and Markushevich theorems and in some particular cases, if such elements are Banach-valued continuous maps on a compact metric space, estimate massivity of their graphs and level sets. 
 \end{abstract}

\date{}

\maketitle

\sect{Introduction}
Let $A$ be an infinite dimensional separable Banach space over field $\mathbb F$ ($\cong\mathbb R$ or $\mathbb C$) and $V_0 \subsetneq V_1 \subsetneq \cdots \subsetneq V_i \subset \cdots \subset A$ be an {\em approximating family} of subspaces, i.e., $V:=\cup_i V_i$ is dense in $A$  and ${\rm dim}_{\mathbb F} V_i<\infty$ for all $i$. Let $B$ be a Banach space over $\mathbb F$ and $(A\hat{\otimes}_\alpha B,\|\cdot\|_{A\hat\otimes_\alpha B})$ a tensor product of Banach spaces  defined by means of a reasonable crossnorm  $\alpha$ on $A\otimes B$ (see, e.g., \cite{R} and subsection~2.1 below for basic results and definitions). 

For $x\in A\hat{\otimes}B$, we set
\begin{equation}\label{ba}
E_n(x):=\inf_{h\in V_n\otimes B}\|x-h\|_{A\hat\otimes_\alpha B}.
\end{equation}
Clearly, $\{E_n(x)\}_{n\in\mathbb Z_+}\subset\mathbb R_+$ is a nonincreasing sequence tending to $0$ as $n\rightarrow\infty$. 
\smallskip
\begin{D}\label{def1}
Let functions $\kappa, \varphi :\mathbb N\rightarrow (0,\infty)$ be such that $\sum_{n=1}^\infty\kappa(n)<\infty$ and  $\varphi$ is nonincreasing with $\lim_{n\rightarrow\infty}\varphi(n)=0$. Adapting the terminology of Beurling \cite{Be}, we say that an element $x\in  A\hat{\otimes}_\alpha B$ belongs to the first Bernstein class 
$\mathfrak{B}_\kappa^*((A,V);B;\alpha)$ if
\begin{equation}\label{beurling}
\prod_{n=1}^\infty (E_n(x))^{\kappa(n)}=0.
\end{equation}

An element $x\in  A\hat{\otimes}_\alpha B$ is said to belong to the second Bernstein class $\mathfrak{B}_{\varphi}((A,V);B;\alpha)$ if 
\begin{equation}\label{scale}
\varliminf_{n\rightarrow\infty} (E_n(x))^{\varphi(n)}<1.
\end{equation}
\end{D}
In what follows we call $\kappa$ in \eqref{beurling} a weight function and $\varphi$ in \eqref{scale} a scale function.

We set  
\[
\begin{array}{c}
\displaystyle
\mathfrak{B}((A,V);B;\alpha):=\mathfrak{B}_{\varphi_0}((A,V);B;\alpha),\quad {\rm where} \quad\varphi_0(n):=\frac 1n,\quad n\in\mathbb N,\medskip\\
\displaystyle
\mathfrak{B}^*((A,V);B;\alpha):=\mathfrak{B}_{\kappa_0}^*((A,V);B;\alpha),\quad {\rm where}\quad \kappa_0(n):=\frac{1}{n^2},\quad n\in\mathbb N.
\end{array}
\]
(In Theorem \ref{prop2.2}  below we show that $\mathfrak{B}((A,V);B;\alpha)\subset \mathfrak{B}^*((A,V);B;\alpha)$.)  

Also, if $B=\mathbb F$, then $A\hat\otimes_\alpha B=A$, and we set 
\[
\mathfrak{B}_\varphi(A,V):=\mathfrak{B}_\varphi((A,V);B;\alpha)\quad {\rm and}\quad \mathfrak{B}_\kappa^*(A,V):=\mathfrak{B}_\kappa^*((A,V);B;\alpha).
\]

For instance, if $A=C(I)$ is the Banach space of complex-valued continuous functions on a compact interval $I\subset\mathbb R$, $V$ is the restriction to $I$ of the space of univariate complex polynomials so that $V_i$ consists of restrictions of polynomials of degree at most $i$, and $B=\mathbb C$, then $\mathfrak{B}(A,V)$ coincides with $\mathfrak{B}(I)$, the {\em set of quasianalytic functions in the sense of Bernstein} \cite{B1, B2}, having the following properties:\smallskip

\noindent (1) (Szmuszkowicz\'{o}wna \cite{S}) \smallskip

{\em The zero locus of each $f\not\equiv 0$ in $\mathfrak{B}(I)$ has transfinite diameter $0$.} 

\smallskip

\noindent (2) (Mazurkiewicz \cite{Ma})\smallskip

 {\em $\mathfrak{B}(I)$ is comeager (i.e., the intersection of countably many sets with dense interiors)  in $C(I)$.}
 
 \smallskip
 
 \noindent (3) (Markushevich \cite{M}) \smallskip
 
 {\em Each $f\in C(I)$ can be written as $f=f_1+f_2$, where 
 $f_1,f_2\in\mathfrak{B}(I)$.}

 \smallskip
 
 \noindent (In fact, as has been observed by Ple\'{s}niak \cite{P1},  (3) is a simple corollary of (2).)\smallskip
 
 Similarly to property (1) functions of the classical first Bernstein class (introduced by Beurling) satisfy:\smallskip

\noindent  (4)   (Beurling \cite{Be}) \smallskip

{\em The zero locus of each $f\not\equiv 0$ in $\mathfrak{B}^*(C(I),V)$ has Lebesgue measure $0$. }\smallskip

 In the classical setting of univariate complex-valued continuous functions, the most general quasianalytc class (including classes $\mathfrak B(I)$ and $\mathfrak{B}^*(C(I),V)$) was introduced and studied by Beurling \cite{Be}. In turn, a unified exposition of the basic concepts and facts concerning quasianalytic functions in the sense of Bernstein defined on compact subsets of $\mathbb C^n$ was presented by Ple\'{s}niak in \cite{P2}. In this paper we continue this line of research and study some general properties of classes $\mathfrak{B}_\varphi((A,V);B;\alpha)$ and  $\mathfrak{B}_\kappa^*((A,V);B;\alpha)$ focusing on the properties similar to (1)--(4) above.\smallskip

The paper is organized as follows. Section 2 is devoted to the formulation and discussion of main results accompanied with the corresponding examples and open problems. Specifically, in subsection~2.1 we formulate some set theoretic and algebraic properties of Bernstein classes $\mathfrak{B}_\kappa^*$ and $\mathfrak{B}_\varphi$ (e.g., we show that sets $\mathscr B^*$ and $\mathscr B$ of all pairwise distinct first and second Bernstein classes form isomorphic partially ordered abelian semigroups with addition given by the union of classes). Subsection~2.2 deals with noncommutative Banach-valued analogs of Markushevich and Mazurkiewicz theorems. For example, our results imply that every invertible continuous matrix function $f:[0,1]\rightarrow G$, where $G$ is  a matrix Lie subgroup of $GL_n(\mathbb F)$, can be presented as $f=f_1\cdot f_2$, where $f_i, f_i^{-1}: [0,1]\rightarrow G$, $i=1,2$,  are  matrix functions with entries quasianalytic in the sense of Bernstein if $G=GL_n(\mathbb F)$ (Theorem \ref{te1}) or belonging to the first Bernstein class $\mathfrak{B}^*(C(I),V)$ in the general case (Theorem \ref{mar2}). Subsection~2.3 is devoted to some properties of elements of classes  $\mathfrak{B}_\varphi$ and $\mathfrak{B}_\kappa^*$ for  spaces $C(M;B)$ of continuous $B$-valued functions on a compact metric space $(M,d)$. For instance, we show that for $M=[0,1]\subset\mathbb R$  graphs of $B$-valued continuous functions  in the first Bernstein class 
$\mathfrak{B}^*$ has Hausdorff dimension $1$ (while  if ${\rm dim}_{\mathbb F} B=\infty$ there are functions with graphs having Hausdorff dimension $\infty$, say, graphs of Peano curves; note that,  due to results of subsection~2.2,  such functions can be presented as sums of  two functions in the second Bernstein class $\mathfrak{B}$). In this setting,  we also formulate some results on massivity of level sets of functions in $\mathfrak{B}_\varphi$ and $\mathfrak{B}^*_\kappa$. Sections 3 -- 5 contain proofs of our main results.

\sect{Main Results}
\subsection{Basic Properties}
Recall that a tensor product $A\hat\otimes_\alpha B$ is the completion of the algebraic tensor product $A\otimes B$ with respect to a {\em reasonable crossnorm} $\alpha$, i.e., such that
\[
\begin{array}{c}
\displaystyle
\alpha(a\otimes b):=\|a\|_A\cdot\|b\|_B\quad {\rm for\ all}\quad a\in A,\ b\in B\quad {\rm and}\medskip
\\
\displaystyle
\quad  \alpha'(a^*\otimes b^*)=\|a^*\|_{A^*}\cdot\|b^*\|_{B^*}\quad {\rm for\ all}\quad a\in A^*,\ b^*\in B^*,
\end{array}
\]
where $\alpha'$ is the dual norm of $\alpha$. 

A {\em uniform crossnorm} is  a method of ascribing to each pair $(X, Y)$ of Banach spaces a reasonable crossnorm  $\alpha$ on $X \otimes Y$ so that if $X, W, Y, Z$ are arbitrary Banach spaces then for all bounded linear operators $S: X \rightarrow W$ and $T: Y \rightarrow Z$ the linear operator $S \otimes T : X \otimes_\alpha Y \rightarrow W \otimes_\alpha Z$ is bounded and $\|S \otimes T\| \leq \|S\|\cdot \|T\|$.\smallskip

Our first result follows directly from Definition \ref{def1}.
\begin{Proposition}\label{prop1}
Suppose $((A,V);B)$ and $((A',V');B')$ are triples of Banach spaces and their subspaces subject to Definition \ref{def1} and 
$S:A\rightarrow A'$, $T:B\rightarrow B'$ are bounded linear operators such that $S(V_n)\subset V_n'$ for all sufficiently large $n\in\mathbb N$.
Then for a uniform crossnorm $\alpha$ and all scale and weight functions $\varphi$ and $\kappa$ the linear operator $S\otimes T$ maps  $\mathfrak{B}_{\varphi}((A,V);B;\alpha)$ into  $\mathfrak{B}_{\varphi}((A',V');B';\alpha)$ and $\mathfrak{B}_\kappa^*((A,V);B;\alpha)$ into  $\mathfrak{B}_\kappa^*((A',V');B';\alpha)$.
\end{Proposition}
In particular, for $A'=A$, $V'=V$ and $B'=\mathbb F$ each bounded linear functional $h\in B^*$ determines the bounded linear map ${\rm Id}\otimes h: A\hat\otimes_\alpha B\rightarrow A\,  (= A\otimes \mathbb F)$ which sends  $\mathfrak{B}_{\varphi}((A,V);B;\alpha)$ into  $\mathfrak{B}_{\varphi}(A,V)$ and  $\mathfrak{B}_\kappa^*((A,V);B)$ into  $\mathfrak{B}_\kappa^*(A,V)$.\smallskip

Next, we describe some set theoretic and algebraic properties of classes $\mathfrak{B}_\varphi$ and $\mathfrak{B}_\kappa^*$.

Let $\mathscr C$ be the set of all quadruples $((A,V);B;\alpha)$, where $A,V,B$ are Banach spaces and their subspaces subject to Definition \ref{def1}, and $\alpha:=\alpha(A,B)$ are reasonable crossnorms on $A\otimes B$.
\begin{Th}\label{prop2.1}
For every $Q=((A,V);B;\alpha)\in \mathscr C$ we have
\smallskip

\noindent {\rm (1)} $\mathfrak{B}_{\varphi}(Q)$ is comeager in $A\hat\otimes_\alpha B$ for each scale function $\varphi$.\smallskip

\noindent {\rm (2)}
\[
A\hat\otimes_\alpha B=\bigcup_\varphi \mathfrak{B}_\varphi(Q)\quad {\rm and} \quad V\otimes B=\bigcap_\varphi \mathfrak{B}_\varphi(Q),
\]
where the union and the intersection are taken over all scale functions $\varphi$. \smallskip

\noindent In turn, $A\hat\otimes_\alpha B\setminus\cup_{\varphi\in\Phi} \mathfrak{B}_\varphi(Q)\ne\emptyset$ for any countable subset $\Phi$ of scale functions.\medskip

\noindent {\rm (3)} For a countable subset $\Phi$ of scale functions there exist scale functions $\varphi_{lb}$ and $\varphi_{ub}$ such that 
\[
\qquad\quad \mathfrak{B}_{\varphi_{lb}}(Q)\subset\bigcap_{\varphi\in\Phi} \mathfrak{B}_\varphi(Q)\subset \bigcup_{\varphi\in\Phi} \mathfrak{B}_\varphi(Q)\subset \mathfrak{B}_{\varphi_{ub}}(Q).
\]

\noindent {\rm (4)} $\mathfrak{B}_{\varphi'}(Q)\subset\mathfrak{B}_{\varphi}(Q)$ if and only if $\varphi'\le c\cdot\varphi$ for some $c\in (0,\infty)$. In particular,
$\mathfrak{B}_{\varphi^2}(Q)\subsetneq\mathfrak{B}_{\varphi}(Q)\subsetneq \mathfrak{B}_{\sqrt{\varphi}}(Q)$.
\end{Th}

Part (1) of the theorem is the analog of property (2) of the Introduction. It reflects the fact that any ``generic metric property'' of $A\hat\otimes_\alpha B$ is also generic for each class $\mathfrak B_{\varphi}(Q)$ (cf. Theorem \ref{te1} below).

In turn, part (4) implies that sets $\mathfrak{B}_{\varphi}(Q)$ and $\mathfrak{B}_{\varphi'}(Q)$ coincide if and only if there exist $c_1,c_2\in (0,\infty)$ such that $c_1\cdot\varphi\le \varphi'\le c_2\cdot\varphi$.   Such $\varphi$ and $\varphi'$ are called equivalent. So $\mathfrak{B}_{\varphi}(Q)$ is uniquely determined by the equivalence class $\{\varphi\}$ of $\varphi$ and we may write $\mathfrak{B}_{\{\varphi\}}(Q)$ instead of $\mathfrak{B}_{\varphi}(Q)$. 

Let $\mathscr B (Q)$ be the collection of all pairwise distinct
sets $\mathfrak{B}_{\{\varphi\}}(Q)$. We define the partial order on $\mathscr B (Q)$ by means of inclusion of sets.  As a consequence of Theorem \ref{prop2.1} we obtain:
\begin{C}\label{lat}
 $(\mathscr{B}(Q),\subseteq )$ is a lattice and
 \[
 \begin{array}{l}
 \displaystyle
 \mathfrak{B}_{\{\varphi_1\}}(Q)\vee \mathfrak{B}_{\{\varphi_2\}}(Q)=\mathfrak{B}_{\{\max(\varphi_1,\varphi_2)\}}(Q)= \mathfrak{B}_{\{\varphi_1\}}(Q)\cup \mathfrak{B}_{\{\varphi_2\}}(Q);\\
 \\
 \displaystyle
 \mathfrak{B}_{\{\varphi_1\}}(Q)\wedge \mathfrak{B}_{\{\varphi_2\}}(Q)=\mathfrak{B}_{\{\min(\varphi_1,\varphi_2)\}}(Q)\, \bigl(\subseteq   \mathfrak{B}_{\{\varphi_1\}}(Q)\cap\mathfrak{B}_{\{\varphi_2\}}(Q)\bigr).
 \end{array}
 \]
\end{C}
Parts (2) and (3) of Theorem \ref{prop2.1} imply that any countable two-sided chain of elements of $\mathscr{B}(Q)$ has upper and lower bounds. However, $\mathscr{B}(Q)$ does not contain neither maximal nor minimal elements. Hence, Zorn's lemma is not applicable and, in particular, there exist uncountable chains in $\mathscr{B}(Q)$ which do not have upper or lower bounds. 

Note that $\mathscr{B}(Q)$ has the structure of a commutative semiring with addition $\cup\, (=\vee)$ and multiplication given by 
$\mathfrak{B}_{\{\varphi_1\}}(Q)\cdot \mathfrak{B}_{\{\varphi_2\}}(Q)=\mathfrak{B}_{\{\varphi_1\cdot\varphi_2\}}(Q)$ induced by addition and multiplication on the set of scale functions.

Let $\mathcal S$ be the set of equivalence classes of scale functions. We introduce the partial order on $\mathcal S$ writing $\{\varphi_1\}\le\{\varphi_2\}$ if and only if $\varphi_1\le c\cdot\varphi_2$ for some $c>0$. In addition, we regard $\mathcal S$ as a commutative semiring with addition $\{\varphi_1\}+\{\varphi_2\}=\{\max(\varphi_1,\varphi_2)\}$ and multiplication $\{\varphi_1\}\cdot \{\varphi_2\}=\{\varphi_1\cdot\varphi_2\}$. Then Theorem \ref{prop2.1} and Corollary \ref{lat} yield:
\begin{C}\label{iso}
For each $Q\in\mathscr C$ map $I_Q: (\mathcal S,+,\cdot,\le)\rightarrow (\mathscr{B}(Q),\cup,\cdot,\subseteq)$, $I_Q(s):=\mathfrak{B}_s(Q)$,  is an isomorphism of semirings preserving partial orders.
\end{C}
In particular, $\mathscr{B}(Q)$ as well as $\mathcal S$ has the cardinality of the continuum.\smallskip 

Let us formulate analogous results for the first Bernstein classes.  

To this end, for a weight function $\kappa$ we define the scale function $\Sigma(\kappa)$ by the formula:
\begin{equation}\label{equ2.1}
\Sigma(\kappa)(n):=\sum_{i=n}^\infty \kappa(i),\quad n\in\mathbb N.
\end{equation}
\begin{Th}\label{prop2.2}
For every $Q=((A,V);B;\alpha)\in \mathscr C$ we have\smallskip

\noindent {\rm (1)} 
$\mathfrak{B}_{\Sigma(\kappa)}(Q)\subsetneq \mathfrak{B}^*_{\kappa}(Q)$
and is maximal among all subsets in  $\mathscr{B}(Q)$ containing in $\mathfrak{B}^*_{\kappa}(Q)$.\smallskip

\noindent {\rm (2)} $\mathfrak{B}_{\Sigma(\kappa_1)}(Q)=\mathfrak{B}_{\Sigma(\kappa_2)}(Q)$ if and only if $\mathfrak{B}^*_{\kappa_1}(Q)=\mathfrak{B}^*_{\kappa_2}(Q)$.\smallskip

\noindent {\rm (3)} For each scale function $\varphi$ there is a weight function $\kappa$ such that $\mathfrak{B}_{\Sigma(\kappa)}(Q)=\mathfrak{B}_{\varphi}(Q)$.\smallskip

\noindent {\rm (4)} 
\[
\mathfrak{B}^*_\kappa(Q)=\bigcap_{\varphi\in\Phi_\kappa}\mathfrak{B}_{\varphi}(Q),
\]
where $\Phi_\kappa$ is the class of scale functions $\varphi$ such that \ $\displaystyle \sum_{n=1}^\infty \frac{\kappa(n)}{\varphi(n)}<\infty$.\medskip 
\end{Th}
\begin{R}\label{rem2.6}
{\rm Since $\mathscr{B}(Q)$ is a lattice, the maximal element in Theorem \ref{prop2.2}\,(1) is unique.
}
\end{R}
From this theorem we obtain the following result.
\begin{C}\label{cor2.4}
Sets $\mathfrak{B}_\kappa^*(Q)$ satisfy the properties similar to (1)--(3) of Theorem \ref{prop2.1} with $\mathfrak{B}$ replaced by $\mathfrak{B}^*$.
\end{C}
\begin{E}\label{ex2.5}
{\rm If $\kappa_0(n):=\frac{1}{n^2}$, $n\in\mathbb N$, then $\bigl\{\Sigma(\kappa_0)\bigr\}=\{\varphi_0\}$, where $\varphi_0(n)=\frac{1}{n}$, $n\in\mathbb N$. Hence, Theorem \ref{prop2.2} implies that $\mathfrak{B}(Q)\subsetneq \mathfrak{B}^*(Q)$
and $\mathfrak{B}(Q)$ is maximal among all subsets in  $\mathscr{B}(Q)$ containing in $\mathfrak{B}^*(Q)$ (see the Introduction for the notation).  
}
\end{E}
By $\mathscr{B}^*(Q)\subset 2^{A\hat\otimes_\alpha B}$ we denote the collection of all pairwise distinct subsets $\mathfrak{B}^*_{\kappa}(Q)$ equipped with the partial order defined by inclusion of sets. Also, we equip $\mathscr{B}^*(Q)$ with the structure of an abelian semigroup with addition given by
\[
\mathfrak{B}^*_{\kappa_1}(Q)+ \mathfrak{B}^*_{\kappa_2}(Q)=\mathfrak{B}^*_{\max(\kappa_1,\kappa_2)}(Q).
\]
The next result shows that this operation is well-defined $\bigl($i.e., if $\mathfrak{B}^*_{\kappa_i}(Q)=\mathfrak{B}^*_{\kappa_i'}(Q)$, $i=1,2$, then $\mathfrak{B}^*_{\max(\kappa_1,\kappa_2)}(Q)=\mathfrak{B}^*_{\max(\kappa_1',\kappa_2')}(Q) \bigr)$ and
coincides with the union of sets.
\begin{Proposition}\label{prop2.5}
We have 
\[
\mathfrak{B}^*_{\max(\kappa_1,\kappa_2)}(Q)=\mathfrak{B}^*_{\kappa_1}(Q)\cup \mathfrak{B}^*_{\kappa_2}(Q)=\mathfrak{B}^*_{\kappa_1}(Q)\vee\mathfrak{B}^*_{\kappa_2}(Q).
\]
\end{Proposition}
As the consequence of Theorem \ref{prop2.2} we obtain:
\begin{C}\label{cor2.10}
Map $\Sigma$ induces the isomorphism of abelian semigroups
\[
\Sigma_*:(\mathscr{B}^*(Q),\cup,\subseteq)\rightarrow  (\mathscr B(Q),\cup,\subseteq),\qquad \Sigma_*(U):=\mathfrak B_{\Sigma(\kappa)}(Q),\quad U=\mathfrak B^*_\kappa(Q),
\]
preserving partial orders such that $\Sigma_*(U)\subsetneq U$ and is maximal among all subsets in  $\mathscr{B}(Q)$ containing in $U$.
\end{C}
\begin{R}
{\rm  (1) This corollary implies that $\Sigma_*^{-1}(W)$ is minimal among all subsets in  $\mathscr{B}^*(Q)$ containing $W\in \mathscr{B}(Q)$ and $(\mathscr{B}^*(Q),\subseteq)$ is a lattice, where  the operation$\,\vee=\cup$, see Proposition \ref{prop2.5}, and 
\[
U_1\wedge U_2:=\Sigma_*^{-1}\bigl(\Sigma(U_1)\wedge\Sigma(U_2)\bigr),\quad  U_i\in \mathscr{B}^*(Q),\quad  i=1,2.
\]
Note that if $U_i=\mathfrak B^*_{\kappa_i}(Q)$, $i=1,2$, then $U_1\wedge U_2=\mathfrak B^*_\kappa(Q)$, where
\[
\kappa(n)=\min\bigl(\Sigma(\kappa_1)(n),\Sigma(\kappa_2)(n)\bigr)-
\min\bigl(\Sigma(\kappa_1)(n+1),\Sigma(\kappa_2)(n+1)\bigr),\quad n\in\mathbb N.
\]
(2) It is worth mentioning that $\mathfrak B_\varphi$ and $\mathfrak B_\kappa^*$ can be regarded as functors from the category $\mathscr C_u$ of quadruples $Q=((A,V);B;\alpha)\in\mathscr C$ with $\alpha$ being uniform crossnorms and with morphisms as in Proposition \ref{prop1} into the category $\mathscr T$ of power sets of tensor products $A\hat\otimes_\alpha B$ whose images form subcategories consisting of the corresponding Bernstein classes. In this terminology, $\Sigma_*$ becomes a functor establishing isomorphism between subcategories $\mathscr B^*$ and $\mathscr B$ of $\mathscr T$ consisting of elements of $\mathscr B^*(Q)$ and $\mathscr B(Q)$.}
\end{R}

Results of subsection~2.1 are proved in Section~3. One of the main ingredients in the proofs is the following version of the classical Bernstein theorem (see subsection~3.1\,(4)):

{\em For each nonincreasing converging to $0$ sequence $\{c_n\}\subset (0,\infty)$ there exists an element $x\in A\hat\otimes_\alpha B$ such that $E_n(x)=c_n$ for all $n\in\mathbb N$.}

\subsection{Markushevich and Mazurkiewicz type theorems} 
First, we formulate some analogs of the Markushevich theorem \cite{M}.

Recall that if $A$ and $B$ are unital Banach algebras, then $A\hat{\otimes}_\alpha B$ has the natural structure of a Banach algebra such that
\[
(a_1\otimes b_1)\cdot (a_2\otimes b_2)=(a_1\cdot a_2)\otimes (b_1\cdot b_2)\quad {\rm for\ all} \quad a_1,a_2\in A,\ b_1,b_2\in B.
\]
By $(A\hat{\otimes}_\alpha B)^{-1}$ we denote the group of invertible elements of the Banach algebra $A\hat{\otimes}_\alpha B$.
\begin{Th}\label{te1}
\begin{itemize}
\item[(1)] Each $x\in A\hat{\otimes}_\alpha B$ can be written as $x=x_1+x_2$, where  $x_1,x_2\in \mathfrak{B}_{\varphi}((A,V);B;\alpha)$.

\item[(2)]
Suppose $A$ and $B$ are unital Banach algebras. Then each $x\in (A\hat{\otimes}_\alpha B)^{-1}$ can be written as $x=x_1\cdot x_2$, where $x_1, x_1^{-1}, x_2, x_2^{-1}\in (A\hat{\otimes}_\alpha B)^{-1}\cap \mathfrak{B}_{\varphi}((A,V);B;\alpha)$.
\end{itemize}
\end{Th}
The particular case of Theorem \ref{te1} for quasianalytic functions in the sense of Bernstein defined on compact subsets of $\mathbb C^n$ was originally proved in \cite{P2}. Due to Theorem \ref{prop2.2}, similar results are valid also for the first Bernstein classes $\mathfrak B^*_\kappa ((A,V);B;\alpha)$.
\begin{E}\label{ex1}
{\rm Let $A=W(\mathbb T)$ be the {\em Wiener algebra} of $\mathbb F$-valued continuous functions on the unit circle $\mathbb T:=\{z\in\mathbb C\, :\, |z|=1\}$ with absolutely convergent Fourier series. The projective tensor product $W(\mathbb T)\hat\otimes_\pi B$ consists of continuous $B$-valued functions
$f:\mathbb T\rightarrow B$ such that
\begin{equation}\label{wiener}
f(t)=\sum_{n=-\infty}^\infty \hat f_n e^{i nt},\quad \hat f_n\in B,\ t\in\mathbb T,\quad {\rm and}\quad \|f\|:=\sum_{n=-\infty}^\infty\|\hat f_n\|_B<\infty .
\end{equation}
Let $V=T\subset W(\mathbb T)$ be the space of $\mathbb F$-valued trigonometric polynomials and $V_i=T_i\subset T$, $i\in\mathbb Z_+$, be the space of polynomials of degree at most $i$. Then one easily shows, see \eqref{ba}, that for $f$ as in \eqref{wiener}
\[
E_i(f)=\sum_{|n|>i}\|\hat f_n\|_B.
\]
Thus, $\mathfrak{B}_{\varphi}((W(\mathbb T),T);B;\pi)$ consists of functions $f$ such that
\[
\varliminf_{i\rightarrow\infty}\left(\sum_{|n|>i}\|\hat f_n\|_B\right)^{\varphi(i)}<1.
\]

Note that if a scale function $\varphi$ decreases sufficiently fast, then the image of a nonconstant function $f\in\mathfrak{B}_{\varphi}((W(\mathbb T),T);B;\pi)$ has Hausdorff dimension $1$, see Theorem \ref{te1.9} below,  while there exist $f\in W(\mathbb T)\hat\otimes_\pi B$, where ${\rm dim}_{\mathbb F}B\ge 2$,
whose images are of Hausdorff dimension $>1$ (see, e.g., \cite[p.\,136]{LG}).\smallskip

Let us formulate Theorem \ref{te1}\,(2) for $B=M_k(\mathbb F)$, the Banach algebra of $k\times k$ matrices with entries in $\mathbb F$. In this case,
$B^{-1}=GL_k(\mathbb F)\subset M_k(\mathbb F)$ is the group of invertible matrices and $W(\mathbb T)\hat\otimes_\pi M_k(\mathbb F)=W(\mathbb T)\otimes M_k(\mathbb F)$ consists of matrix-valued functions on $\mathbb T$ with entries in $W(\mathbb T)$. As follows from the classical Wiener theorem, $(W(\mathbb T)\otimes_\pi M_k(\mathbb F))^{-1}$ consists of functions in $W(\mathbb T)\otimes_\pi M_k(\mathbb F)$ with images in 
$GL_k(\mathbb F)$. Hence, each such a function $g$ can be written as $g=g_1\cdot g_2$, where $g_i^{\pm 1}\in  \mathfrak{B}_{\varphi}((W(\mathbb T),T);M_k(\mathbb F);\pi)$, $i=1,2$.
}
\end{E}

Next, we formulate some versions of Theorem \ref{te1}\,(2) for $A$
being a {\em semisimple commutative unital complex Banach algebra}.
Then the {\em Gelfand transform} $\hat\, : A\rightarrow C(M_A)$, where $M_A:=\{\xi\in {\rm Hom}(A,\mathbb C)\setminus\{0\}\}$ is the maximal ideal space of $A$ equipped with the {\em Gelfand topology}, is an injective nonincreasing norm morphism of algebras. Without loss of generality we identify $A$ with its image under $\, \hat{}\, $, so that $A$ consists of complex continuous functions on the compact Hausdorff space $M_A$. If, in addition, $A$ is invariant under the standard operation of complex conjugation on $C(M_A)$, then by $A_{\mathbb R}$ we denote the subalgebra of $A$ consisting of real functions, so that $A=A_\mathbb R\oplus \sqrt{-1} \cdot A_\mathbb R$. In this case, we assume also that each $V_i\subset V$ is invariant under complex conjugation and set $V_{i,\mathbb R}=V_i\cap A_{\mathbb R}$, $V_{\mathbb R}=V\cap A_{\mathbb R}$. 
\medskip

\noindent (*) In what follows $A_{\mathbb F}$, $V_{\mathbb F}$ and $V_{i,\mathbb F}$ stand for $A$, $V$ and $V_i$ if $\mathbb F=\mathbb C$ and for $A_{\mathbb R}$, $V_{\mathbb R}$ and $V_{i,\mathbb R}$ if $\mathbb F=\mathbb R$.
Additionally, we assume that $V\subset A$ is a {\em filtered unital algebra}, that is,
$V_i\cdot V_j\subset V_{i+j}$ for all $i,j\in \mathbb Z_+$.
\medskip

One naturally identifies  $A_{\mathbb F}\otimes\mathbb F^N$ with the space of continuous maps $g=(g_1,\dots, g_N):M_A\rightarrow\mathbb F^N$ with all $g_i\in A$.  Since all reasonable norms on $A_{\mathbb F}\otimes\mathbb F^N$ are equivalent, the corresponding Bernstein classes are independent of their choice  and we write $((A,V);\mathbb F^N)$ instead of $((A,V);\mathbb F^N;\alpha)$ in their definitions.

For convenience, we equip $A_{\mathbb F}\otimes\mathbb F^N$ with norm
\[
\|g\|:=\left(\sum_{i=1}^N\|g_i\|_A^2\right)^{\frac 12}.
\]
For $X\subset\mathbb F^N$ by $X(A_\mathbb F)$ we denote the subset of $g\in A_{\mathbb F}\otimes\mathbb F^N$ with images in $X$. It is easily seen, using that the Gelfand transform is a nonincreasing norm morphism of algebras, that if $X$ is closed, then $X(A_\mathbb F)$ is a closed subset of the Banach space $(A_{\mathbb F}\otimes\mathbb F^N,\|\cdot\|)$.\smallskip

Next, recall that the {\em polynomially convex hull} $\widehat K$ of a subset $K\Subset\mathbb C^N$ is determined by
\[
\widehat K=\left\{z\in \mathbb C^N\, :\, |p(z)|\le\sup_K |p|\ \ {\rm for\ all}\ p\in\mathcal P(\mathbb C^N)\right\}.
\]
\noindent Here $\mathcal P(\mathbb C^N)$ stands for the space of holomorphic polynomials on $\mathbb C^N$. 

Set $K$ is called {\em polynomially convex} if $\widehat K=K$.

\begin{Th}\label{mazur2}
Suppose that $X_\mathbb F$ is a $\mathbb F$-analytic closed submanifold of a domain $U_\mathbb F\subset \mathbb F^N$ such that $U_\mathbb C$ admits an exhaustion by compact polynomially convex subsets. Then 
\begin{itemize}
\item[(a)] For each scale function $\varphi:\mathbb N\rightarrow (0,\infty)$ such that $\lim_{n\rightarrow\infty}n\cdot\varphi(n)=\infty$ set $\mathfrak{B}_{\varphi}((A,V);\mathbb F^N)\cap X_{\mathbb F}(A_\mathbb F)$ is comeager in $X_{\mathbb F}(A_\mathbb F)$;\smallskip
\item[(b)]  Set $\bigl(\mathfrak{B}^*((A,V);\mathbb F^N)\cap \mathfrak{B}_{\tilde\varphi}((A,V);\mathbb F^N)\bigr)\cap X_{\mathbb F}(A_\mathbb F)$, where $\tilde\varphi(n)=\frac{1+\ln n}{n}$, $n\in\mathbb N$, is comeager in $X_{\mathbb F}(A_\mathbb F)$.
\end{itemize}
\end{Th}

In addition, suppose that $X_\mathbb F$ has the structure of a $\mathbb F$-analytic Lie group compatible with the analytic structure induced from $\mathbb F^N$.
Then the group operations on $X_\mathbb F$ induce group operations $\, \cdot\, ,^{-1}$ on the set $C(M_A;X_{\mathbb F})$ of continuous functions on $M_A$ with images in  $X_\mathbb F$.
Using functional calculus for commutative Banach algebras (see, e.g., \cite[Ch.\,3.4]{G}), we show that $X_\mathbb F(A_\mathbb F)\subset C(M_A;X_{\mathbb F})$ is a subgroup (see the proof of Theorem \ref{mar2} below). Now, as the corollary of Theorem \ref{mazur2} we obtain the analog of the Markushevich theorem.
\begin{Th}\label{mar2}
Let $\mathfrak{B}$ be one of the subsets of $X_{\mathbb F}(A_\mathbb F)$ of parts (a) and (b) above. Then
each  $g\in X_\mathbb F(A_\mathbb F)$ can be written as $g=g_1\cdot g_2$, where $g_1^{\pm 1}, g_2^{\pm 1}\in \mathfrak{B}$.
\end{Th}

\begin{E}\label{e1.7}
{\rm Let $C^n[0,1]$ be the Banach algebra of $n$-times continuously differentiable complex-valued functions on $[0,1]$ with pointwise multiplication and norm
\[
\|f\|:=\sum_{k=0}^n\frac{1}{k!}\sup_{0\le t\le 1}|f^{(k)}(t)|.
\]
Let $V=\mathcal P\subset C^n[0,1]$ be the space of restrictions to $[0,1]$ of complex polynomials on $\mathbb R$ and $V_i=\mathcal P_i$, $i\in\mathbb Z_+$, be the space of restrictions to $[0,1]$ of polynomials of degree at most $i$. One easily checks that $\mathfrak{B}_\varphi((C^n[0,1],\mathcal P);\mathbb C^N)$ and $\mathfrak{B}_\kappa^*((C^n[0,1],\mathcal P);\mathbb C^N)$ consist of $\mathbb C^N$-valued functions $f$ on $[0,1]$ having continuous derivatives of all orders $\le n$ such that $f^{(n)}$ belong to $\mathfrak{B}_\varphi((C[0,1],\mathcal P);\mathbb C^N)$ and $\mathfrak{B}_\kappa^*((C[0,1],\mathcal P);\mathbb C^N)$, respectively.

Let $U_n\subset GL_n(\mathbb C)$ be the unitary group. Since as a real Lie group $U_n$ is isomorphic to a Lie subgroup of the orthogonal group $O_{2n}\subset GL_{2n}(\mathbb R)$, Theorem \ref{mar2} is applicable. Thus, under the assumptions of the theorem, each element $g$ of the group $C^n([0,1];U_n)$ of $n$-times continuously differentiable functions on $[0,1]$ with values in $U_n$ can be written as $g=g_1\cdot g_2$, where $g_1^{\pm 1}, g_2^{\pm 1}\in \mathfrak{B}$ and $\mathfrak{B}$ is either $\mathfrak{B}_\varphi((C^n[0,1],\mathcal P);M_n(\mathbb C))\cap C^n([0,1];U_n)$ or $\bigl(\mathfrak{B}^*((C^n[0,1],\mathcal P);M_n(\mathbb C))\cap \mathfrak{B}_{\tilde\varphi}((C^n[0,1],\mathcal P);M_n(\mathbb C)) \bigr)\cap C^n([0,1];U_n)$.
}
\end{E}
Similarly, Theorem \ref{mar2} can be applied to other matrix Lie subgroups of $GL_n(\mathbb F)$.
\begin{R}\label{rem1.8}
{\rm The natural problem arising with regard to Theorem \ref{mar2} is}
\begin{Problem}\label{p2.8}
Characterize matrix Lie subgroups $X_\mathbb F\subset GL_n(\mathbb F)$ for which the statement of the theorem is valid for all spaces $\mathfrak{B}_{\varphi}((A,V);\mathbb F^N)\cap X_{\mathbb F}(A_\mathbb F)$.
\end{Problem}
{\rm As follows from the proof of Theorem \ref{mazur2}, a sufficient condition for $X_\mathbb F$ to satisfy this property is that the set of maps in $C(M_A;X_\mathbb F)$ whose coordinate functions belong to $V\, (\subset A_\mathbb F)$
is dense in $X_\mathbb F(A_\mathbb F)$. For instance, this is true for $X_\mathbb F=B^{-1}$ being the group of invertible elements of a subalgebra $B\subset M_n(\mathbb F)$ (in this case, $X_\mathbb F(A_\mathbb F)=(A_\mathbb F\otimes B)^{-1}$ is an open subset of $A_\mathbb F\otimes B$), cf. Theorem \ref{te1}.
}
\end{R}

\subsection{Massivity of Graphs and Level Sets of $\mathfrak{B}_\varphi$-functions}
In this part we assume that the Banach space $A$ consists of continuous functions defined on a compact metric space $(M,d)$ and $\|\cdot\|_A\ge \|\cdot\|_{C(M)}$, $V\subset A$ consists of Lipschitz functions on $M$, and $A\hat\otimes_\alpha B$ is equipped with a uniform crossnorm $\alpha$. Since $\alpha\ge \epsilon$, the {\em injective crossnorm}, $A\hat\otimes_\alpha B$ consists of $B$-valued continuous functions on $M$ (see \cite[Th.\,3,\,p.21]{Gr}).  
We study massivity of graphs and level sets of functions in $\mathfrak{B}_\varphi((A,V);B;\alpha)\subset C(M;B)$. 

Recall that the {\em covering number} $Cov(S;\varepsilon)$ of a compact subset $S\subset M$ is defined by 
\[
Cov(S;\varepsilon):=\inf\left[{\rm card}\left\{\{m_i\}\subset S\, :\, S\subset\bigcup_i B_\varepsilon (m_i)\right\}\right],
\]
where $B_r(m)\subset M$ is an open ball of radius $r$ with center $m$.

Covering numbers are qualitative characteristics of the approximation of compact
sets by means of $n$-point subsets, see, e.g., \cite{KT}. It is well known that the function $Cov$ is countably subadditive in $S$ and nonincreasing and continuous from the
right in $\varepsilon$.

Next, for a compact subset $S\subset M$ by $\mathcal M_{V_i}(S)$ we denote the {\em Markov constant} of functions in $V_i$ restricted to $S$ defined by
\[
\mathcal M_{V_i}(S):=\sup_{f\in V_i,\, \|f\|_{C(M)}\le 1}\left\{\sup_{x\ne y}\frac{|f(x)-f(y)|}{d(x,y)}\right\}.
\]
Clearly, $\{\mathcal M_{V_i}(S)\}_{i\in\mathbb N}$ is a nondecreasing sequence of nonnegative numbers, because each ${\rm dim}\,V_i<\infty$. (E.g., in notation of Example \ref{ex1}, the classical Bernstein inequality for derivatives of trigonometric polynomials of degree at most $i$ implies that $\mathcal M_{T_i}(\mathbb T)=i$, $i\in\mathbb Z_+$, here $\mathbb T$ is equipped with the standard geodesic metric.)

In what follows, by $\mathcal H_X^\psi$ we denote the $\psi$-Hausdorff measure on subsets of a metric space $X$ constructed by a {\em gauge function} $\psi$ (i.e., $\psi: [0,\infty)\rightarrow [0,\infty)$ is nondecreasing and $\psi(0)=0$), see, e.g., \cite{Mat} for basic definitions.

For a compact subset $S\subset M$  and $f\in C(M; B)$ by 
\[
\Gamma_f(S):=\{(s,f(s))\, :\, s\in S\}\subset S\times B
\] 
we denote its graph.  In the next result,  we consider $M\times B$ endowed with the metric 
\[
d_{M\times B}\bigl((m_1,b_1),(m_2,b_2)\bigr):=\max\{d(m_1,m_2), \|b_1-b_2\|_B\},\quad (m_i,b_i)\in M\times B,\ i=1,2.
\]
\begin{Th}\label{te1.9}
Suppose a compact subset $S\subset M$ satisfies 
\begin{equation}\label{eq1.5}
Cov(S;\varepsilon)\leq C\cdot\left(\frac{1}{\varepsilon}\right)^k\quad  for\ some\quad C,k\in (0,\infty)\quad and\ all\quad \varepsilon\in (0, {\rm diam}\, S].
\end{equation}
Then for a scale function $\varphi:\mathbb Z_+\rightarrow (0,\infty)$ and a continuous nondecreasing function $\psi: [0,\infty)\rightarrow [0,\infty)$, $\psi(0)=0$, such that
\begin{equation}\label{eq1.6}
\varlimsup_{n\rightarrow\infty}\mathcal M_{V_n}^k(S)\cdot\psi\left(\mathcal M_{V_n}(S)\cdot\rho^{\frac{1}{\varphi(n)}}\right)=:L(\rho)<\infty\quad for\ all\quad \rho\in (0,1),
\end{equation}
and each $f\in\mathfrak{B}_\varphi((A,V);B;\alpha)$,
\[
{\mathcal H}_{M\times B}^{\psi_k}\bigl(\Gamma_f(S)\bigr)<\infty,\quad where\quad \psi_k(t):=t^k\cdot\psi(t),\quad t\in [0,\infty).
\]
\end{Th}
As was shown in Theorem \ref{prop2.1}, massivity of sets 
$\mathfrak{B}_\varphi((A,V);B;\alpha)$ can be measured by means of equivalence classes $\{\varphi\}$ of scale functions $\varphi$. In turn, Theorem \ref{te1.9} shows that the ``smaller'' class $\{\varphi\}$ (i.e., the faster $\varphi(t)$ tends to $0$ as $t\rightarrow\infty$), the ``thinner'' graphs of functions in it (i.e., the slower $\psi(t)$ tends to $0$ as $t\rightarrow 0^+$). 
\begin{E}\label{ex2.10}
{\rm (A) Suppose a scale function $\varphi:\mathbb Z_+\rightarrow (0,\infty)$ satisfies
\begin{equation}\label{eq1.6'}
\varlimsup_{n\rightarrow\infty}\varphi(n)\cdot \mathcal M_{V_n}(S)=:L<\infty.
\end{equation}
Then condition \eqref{eq1.6} is valid for $\psi(t):=\min\bigl(|\ln t|^{-k}, 1\bigr)$, $t\in [0,\infty)$, with $L(\rho)=\left(\frac{L}{|\ln\rho|}\right)^k$.\smallskip

\noindent (B) Suppose a scale function $\varphi:\mathbb Z_+\rightarrow (0,\infty)$ satisfies
\begin{equation}\label{eq1.7}
\lim_{n\rightarrow\infty}\varphi(n)\cdot \ln \mathcal M_{V_n}(S)=0.
\end{equation}
Then condition \eqref{eq1.6} is valid for $\psi(t):=t^\alpha$, $t\in [0,\infty)$, for all $\alpha>0$ with $L(\rho)=0$.
}
\end{E}
\begin{C}\label{cor1.10}
Under the hypotheses of Theorem \ref{te1.9},
\begin{itemize}
\item[(1)]
For each $f\in\mathfrak{B}_\varphi((A,V);B;\alpha)$,
\[
{\mathcal H}_{B}^{\psi_k}\bigl(f(S)\bigr)<\infty.
\]
\item[(2)]
If $B=\mathbb R^d$, $d\le k$, then for each $f\in\mathfrak{B}_\varphi((A,V);B)$,
\[
 {\mathcal H}_{M}^{\psi_{k,d}}\bigl(f^{-1}(c)\cap S\bigr)<\infty,\quad where\quad \psi_{k,d}(t):=t^{k-d}\cdot\psi(t),\quad t\in [0,\infty),
\]
for a.e. $c\in\mathbb R^d$ with respect to Lebesgue measure on $\mathbb R^d$.
\end{itemize}
\end{C}
\begin{R}\label{rem1.11}
{\rm (1) Condition \eqref{eq1.5} implies that the upper box-counting dimension of $S$,
\[
\overline{{\rm dim}}_B S:=\varlimsup_{\varepsilon\rightarrow 0}\frac{\ln Cov(S;\varepsilon)}{-\ln\varepsilon}\le k.
\]
For subsets $S$ such that $\overline{{\rm dim}}_B S={\rm dim}_H S=k$ Theorem \ref{te1.9} with $\varphi$ satisfying \eqref{eq1.7} shows that ${\rm dim}_{H}\bigl(\Gamma_f(S)\bigr)=k$ for all $f\in \mathfrak{B}_\varphi((A,V);B;\alpha)$; here ${\rm dim}_H$ stands for the Hausdorff dimension. This generalizes the result of Mauldin and Williams \cite{MW} asserting that for a generic real $f\in C[0,1]$ its graph $\Gamma_f([0,1])\subset\mathbb R^2$ has Hausdorff dimension one.\smallskip

\noindent (2) For the class $\mathfrak{B}(I)$ of quasianalytic functions in the sense of Bernstein on a compact interval $I\subset\mathbb R$ (see the Introduction), $\mathcal M_{V_n}(S)\le c_S\cdot n$, $n\in\mathbb Z_+$, for each compact subinterval $S$ in the interior of $I$ (here $c_S$ depends on $S$ only). Then condition \eqref{eq1.6'} is satisfied. Thus, Theorem \ref{te1.9} and Corollary \ref{cor1.10} imply that for $\psi(t)=\min\bigl(|\ln t|^{-1}, 1\bigr)$, $\psi_1(t):=t\cdot\psi(t)$, $t\in [0,\infty)$, and
each real $f\in\mathfrak{B}(I)$, 
\begin{equation}\label{bern}
{\mathcal H}_{\mathbb R^2}^{\psi_1}\bigl(\Gamma_f(S)\bigr)<\infty\quad {\rm and}\quad {\mathcal H}_{\mathbb R}^{\psi}\bigl(f^{-1}(c)\cap S\bigr)<\infty\
\end{equation}
for a.e. $c\in f(S)$ with respect to Lebesgue measure on $\mathbb R$. (The second inequality implies that the transfinite diameter of such set $f^{-1}(c)\cap S$ is zero but not vice versa, see, e.g., \cite{C}.) An interesting question is about the optimality of the gauge function $\psi$:}
\begin{Problem}\label{p2.13}
Are there a real function $f\in \mathfrak{B}(I)$  and a compact subinterval $S$ in the interior of $I$ for which the values of Hausdorff measures in \eqref{bern} are not zeros?
\end{Problem}
\noindent {\rm  (3) According to the classical result of Bernstein (see, e.g., 
\cite[p.~400]{T} and reference therein),  the function
\[
f(x)=\sum_{n=1}^\infty \frac{\cos \bigl(F(n)\cdot\arccos x\bigr)}{F(n)}, \quad {\rm where}\quad F(0)=1,\quad  F(n+1):=2^{F(n)},\quad n\in\mathbb Z_+,
\]
belongs to $\mathfrak{B}([-1,1])$ and is not differentiable at any point of $[-1,1]$. Therefore its graph $\Gamma_f(S)\subset \mathbb R^2$ over any nontrivial subinterval $S\subset [-1,1]$ has infinite linear measure. In turn, the method of the proof of Theorem \ref{te1.9} (see subsection~5.1 below) applied to $f$ leads to equality $\mathcal H_{\mathbb R^2}^{\psi_{m1}}(\Gamma_f(S))=0$ for all $m\in\mathbb N$; here $\psi_{m1}(t)=\frac{t}{\psi_m(t)}$ and  $\psi_1(t)=\max\bigl(|\log_2 t|,2\bigr)$,  $\psi_{i+1}(t)=\log_2\psi_{i}(t)$, $i\in\mathbb N$, $t\in [0,\infty)$. Thus, graph $\Gamma_f(S)$ is not rectifiable and is much more thinner than the one for a would-be optimal function in Problem \ref{p2.13}.
}

\end{R}
\begin{E}\label{ex1.2}
{\rm Suppose $M\Subset\mathbb R^N$ is a compact set with nonempty interior equipped with the metric induced by the Euclidean metric on $\mathbb R^N$. We consider $A=C(M)$, the set of complex-valued continuous functions on $M$, and $C(M)\hat\otimes_\epsilon B=C(M;B)$, the injective tensor product of $C(M)$ and a complex Banach space $B$. Let $V=\mathcal P(\mathbb R^N)|_M$ be the space of traces of complex-valued polynomials on $\mathbb R^N$ to $M$ and let $V_n=\mathcal P_n(\mathbb R^N)|_M$ consist of traces of polynomials of degree at most $n$. Using Remez and Markov polynomial inequalities one obtains that there exists a constant $c>0$ such that the Markov constant $\mathcal M_{V_n}(M)\le c^{\, n}$ for all $n$. Hence, condition \eqref{eq1.7} is satisfied for each scale function $\varphi$ such that $\lim_{n\rightarrow\infty} n\cdot\varphi(n)=0$. Thus, for such $\varphi$ and all $\psi$ as in Example \ref{ex2.10}, Theorem \ref{te1.9} and Corollary \ref{cor1.10} are valid for functions $f\in \mathfrak B_\varphi((A,V);B;\epsilon)$, that is,
\begin{equation}\label{eq2.7}
\begin{array}{lr}
\displaystyle
{\rm dim}_H(\Gamma_f(M))=N,\quad  {\rm dim}_H(f(M))\le N\quad {\rm and\ if }\quad B=\mathbb R^d,\ d\le N,\ 
\medskip\\
\displaystyle 
{\rm dim}_H(f^{-1}(c))\le N-d\ {\rm for\ a.e.}\ c\in\mathbb R^d\ {\rm with\ respect\ to\ Lebesgue\ measure\ on}\ \mathbb R^d.
\end{array}
\end{equation}

In turn, if $M$ is a {\em Markov set with a finite exponent} 
(that is,  $\mathcal M_{V_n}(M)\le c\cdot n^r$, $n\in\mathbb N$, for some $c,r>0$, see \cite{BP} for basic properties and examples), then statement \eqref{eq2.7} is valid also for $f\in \mathfrak B_\varphi((A,V);B;\epsilon)$ and all $\varphi$ such that $\lim_{n\rightarrow\infty} \ln n\cdot\varphi(n)=0$. (For instance, any fat subanalytic or convex subsets $M\Subset\mathbb R^N$ are Markov as well as any bounded domain with $C^{1,1}$ boundary.)

Note that the last statement in \eqref{eq2.7} is optimal in the sense that it cannot be valid for all $f\in\mathfrak{B}_\varphi((A,V),\mathbb R^d)$ and all $c\in \mathbb R^d$. Indeed, let 
$\varphi(n)=\frac{1}{(\ln(n+1))^2}$, $n\in\mathbb N$. Then 
\[
\lim_{n\rightarrow\infty} \ln n\cdot\varphi(n)=0\quad {\rm  but}\quad \sum_{n=1}^\infty \frac{1}{\varphi(n)\cdot n^2}<\infty.
\]
From the last inequality by the result of Beurling  \cite[Ch.III, Th.1]{Be} one obtains that there exists a not identically zero real function $f\in C(M)$ equals zero on a subset with nonempty interior such that for some $C>0$
\[
E_n(f)\le C\cdot \left(\frac{1}{2}\right)^{\frac{1}{\varphi(n)}}\quad {\rm for\ all}\quad n\in\mathbb N.
\]
Clearly, $f\in \mathfrak{B}_\varphi(A,V)$ and so, by Corollary \ref{cor1.10}, ${\rm dim}_H(f^{-1}(c))\le N-1$ for a.e. $c\in\mathbb R$ with respect to Lebesgue measure on $\mathbb R$. However, ${\rm dim}_H(f^{-1}(0))=N$.
 }
\end{E}
Now we formulate separately versions of Theorem \ref{te1.9} and Corollary \ref{cor1.10} for functions in the first Bernstein classes $\mathfrak{B}_\kappa^*((A,V);B;\alpha)$. 

First, observe that from Theorem \ref{prop2.2}(4)  (by Corollary \ref{cor1.10}  based on condition \eqref{eq1.7})
we obtain immediately:
\begin{C}\label{cor2.13}
Suppose $S\subset M$ satisfies condition \eqref{eq1.5} 
and for some scale function $\varphi\in\Phi_\kappa$, i.e. such that\ $\displaystyle \sum_{n=1}^\infty\frac{\kappa(n)}{\varphi(n)}<\infty$,
\[
\lim_{n\rightarrow\infty}\varphi(n)\cdot \ln \mathcal M_{V_n}(S)=0.
\]
Then for each $f\in \mathfrak{B}_\kappa^*((A,V);B;\alpha)$,
\[
(a)\quad {\rm dim}_H(\Gamma_f(M))\le k,\quad {\rm dim}_H(f(M))\le k;\quad  (b)\quad {\rm dim}_H(f^{-1}(c))\le k-d
\]
for a.e. $c\in\mathbb R^d$ with respect to Lebesgue measure on $\mathbb R^d$ provided that $B=\mathbb R^d$, $d\le k$.
\end{C}
In particular,  this result is valid for functions $f\in \mathfrak B^*((A,V);B;\alpha)$ with $\varphi(n)=\frac{(s+\ln n)^s}{n}$, $n\in\mathbb N$, for a fixed $s>1$. In turn, it can be sharpened for $\{\mathcal M_{V_n}(S)\}_{n\in\mathbb N}$ having polynomial growth.
\begin{Th}\label{te1.13}
Suppose $S\subset M$ satisfies condition \eqref{eq1.5} and for some $s\ge 1$
\begin{equation}\label{eq1.8}
\sup_{n}\left\{\frac{\mathcal M_{V_n}(S)}{n^s}\right\}<\infty.
\end{equation}
Then for a continuous increasing gauge function $\psi$ satisfying
\begin{equation}\label{eq1.9}
\int_0^1\frac{\bigl(\psi(t)\bigr)^{\frac 1s}}{t}\, dt <\infty
\end{equation}
and each $f\in \mathfrak{B}^*((A,V);B;\alpha)$,
\[
{\mathcal H}_{M\times B}^{\psi_k}\bigl(\Gamma_f(S)\bigr)=0,\quad where\quad \psi_k(t):=(t\cdot\psi(t))^k,\quad t\in [0,\infty).
\]

In particular,
\[
{\mathcal H}_{B}^{\psi_k}\bigl(f(S)\bigr)=0,
\]
and, if $B=\mathbb R^d$, $d\le k$,
\[
{\mathcal H}_{M}^{\psi_{k,d}}\bigl(f^{-1}(c)\cap S\bigr)=0,\quad where\quad \psi_{k,d}(t):=t^{k-d}\cdot (\psi(t))^k,\quad t\in [0,\infty),
\]
for a.e. $c\in\mathbb R^d$ with respect to Lebesgue measure on $\mathbb R^d$.
\end{Th}
\begin{E}
{\rm In notation of Example \ref{ex1.2} consider the class
$\mathfrak{B}^*((A,V);B;\epsilon)$, where $A=C(M)$, $M\Subset\mathbb R^N$ is the closure of a bounded domain, $V=\mathcal P(\mathbb R^N)|_M$ and $A\hat\otimes_\epsilon B= C(M;B)$.  It follows from the result of Beurling  \cite[Ch.III, Th.1]{Be}  that Lebesgue measure of each $f^{-1}(c)\subset M$, $c\in {\rm range}(f)\subset B$, $f\in \mathfrak{B}^*((A,V);B;\epsilon)$ is zero. On the other hand, if $M$ is a Markov set with exponent $s$, then condition \eqref{eq1.8} holds. So Theorem \ref{te1.13} in this setting implies for all $\psi$ satisfying \eqref{eq1.9} and $\psi_N(t):=(t\cdot\psi(t))^N$, $\psi_{N,d}(t):=t^{N-d}\cdot (\psi(t))^N$,  $t\in [0,\infty)$,
\[
\begin{array}{l}
{\mathcal H}_{\mathbb R^N\times B}^{\psi_N}\bigl(\Gamma_f(M)\bigr)=0,\quad
{\mathcal H}_{B}^{\psi_N}\bigl(f(M)\bigr)=0\quad {\rm and\ if}\quad B=\mathbb R^d,\ d\le N,\medskip \\ 
{\mathcal H}_{\mathbb R^N}^{\psi_{N,d}}\bigl(f^{-1}(c)\bigr)=0\ {\rm for\ a.e.}\ c\in\mathbb R^d\ {\rm with\ respect\ to\ Lebesgue\ measure\ on}\ \mathbb R^d.
\end{array}
\]
Here we can take, e.g., 
\[
\psi(t):=\frac{1}{\left(\ln\left(\frac 1t\right)\right)^s\cdot \left(\ln \left( \ln \frac 1t\right)\right)^{s+\varepsilon}},\quad \varepsilon>0,\quad {\rm for}\quad t\in (0,e^{-2}].
\]

If $M=[0,1]\subset\mathbb R$, then condition \eqref{eq1.8} holds with $s=1$ for each compact subinterval $S\subset (0,1)$. Thus, for $B=\mathbb R$ and $\psi$ satisfying \eqref{eq1.9} with $s=1$ by Theorem \ref{te1.13} we obtain
\begin{equation}\label{eq2.11}
{\mathcal H}_{\mathbb R^2}^{\psi_1}\bigl(\Gamma_f(M)\bigr)=0\quad {\rm and}\quad 
{\mathcal H}_{\mathbb R}^{\psi}\bigl(f^{-1}(c)\bigr)=0
\end{equation}
for a.e. $c\in\mathbb R$ with respect to Lebesgue measure on $\mathbb R$. Note that the second condition is slightly weaker than the statement  that $f^{-1}(c)$ has transfinite diameter $0$ (see, e.g., \cite{HK}). Thus, the following question seems to be plausible.
\begin{Problem}\label{p2.19}
Is it true that the second condition in \eqref{eq2.11} can be replaced by $f^{-1}(c)$ has transfinite diameter $0$ for a.e. $c\in\mathbb R$? 
\end{Problem}
}
\end{E}

Finally, we formulate some results related to property (1) of the Introduction, see \cite{S}, for quasianalytic functions in the sense of Bernstein. 

Recall that a subset $S\subset\mathbb C^N$ is {\em pluripolar} if there exists a plurisubharmonic function $u\not\equiv -\infty$ on $\mathbb C^N$ such that $u|_{S}=-\infty$. For a compact subset $K\subset\mathbb C^N$  by  $\mathfrak{B}(K)\subset C(K)$  we denote the set of complex quasianalytic functions in the sense of Bernstein.
In our notation, $\mathfrak{B}(K)=\mathfrak{B}_{\varphi_0}(A,V)$, where $A$ is the uniform closure of the algebra $\mathcal P(\mathbb C^N)|_K$ of traces of holomorphic polynomials on $\mathbb C^N$ to $K$, $V=\mathcal P(\mathbb C^N)|_K$ and $V_i=\mathcal P_i(\mathbb C^N)$, $i\in\mathbb N$, consists of polynomials of degree at most $i$.

In \cite[Th.\,3.2]{CLP} Coman, Levenberg and Poletsky proved that the graph of a function in $\mathfrak{B}(I)$, where $I\Subset\mathbb R$ is a compact interval, is pluripolar in $\mathbb C^2$. The proof is based on the analog of the classical Kellogg lemma due to Bedford and Taylor \cite[Th.\,4.2.5]{BT}. The very same method applies to yield the following result.
\begin{Proposition}\label{clp}
Suppose $K\subset\mathbb C^N$  is a compact non-pluripolar set.  Then  for every $f\in\mathfrak{B}(K)$ the graph $\Gamma_f(K)\subset\mathbb C^{N+1}$  of $f$ is pluripolar. 
\end{Proposition}
As a corollary we obtain:
\begin{C}\label{cor2.16}
Suppose $K\subset\mathbb C^N$  is a compact non-pluripolar set.  Then  for every $f\in\mathfrak{B}(K)$ there exists a (possibly empty) polar set $S_f\subset {\rm range}(f)\subset\mathbb C$ such that for each $c\in {\rm range }(f)\setminus S_f$ set $f^{-1}(c)\subset K$ is pluripolar.
\end{C}
\begin{R}
{\rm (1) In many cases $K$ satisfies the property that 
{\em $S_f=\emptyset$ for all nonconstant} $f\in\mathfrak{B}(K)$ 
(see \cite{P2}, \cite{Sk} for the corresponding results). For instance, as follows from \cite[Th.7.3]{P2}, this is true for $K\subset\mathbb F^N$ being the closure of a bounded domain. 

Note that for $\mathbb F=\mathbb C$ set $f^{-1}(c)\subset K$, where $f\in\mathfrak{B}(K)$ is nonconstant, being polar  satisfies $\mathcal H_{\mathbb C^N}^\psi \bigl(f^{-1}(c)\bigr)=0$ for all $\psi$ such that $\int_0^1\frac{\psi(t)}{t^{2N-1}}dt<\infty$, see, e.g., \cite[Ch.5]{HK}. This is stronger than one obtains from Corollary \ref{cor1.10}(2) in this case (for $k=2N$, $d=2$, $\varphi=\varphi_0$, $\psi_{2N,2}(t):=\frac{t^{2N-2}}{\max(|\ln t|^{2N}, 1)}$, $t\in [0,\infty)$, and, say, $S$ -- a closed cube in $\mathbb R^{2N})$. One can ask a similar question for nonconstant $f\in\mathfrak{B}(K)$ for $K$ being the closure of a bounded domain in $\mathbb R^N$:
\begin{Problem}\label{p2.23}
Is it true that $\mathcal H_{\mathbb R^N}^\psi \bigl(f^{-1}(c)\bigr)=0$ for all $\psi$ such that $\int_0^1\frac{\psi(t)}{t^{N-1}}dt<\infty$?
\end{Problem}
In this case Corollary \ref{cor1.10}(2) with $k=N$, $d=1$, $\varphi=\varphi_0$ and $S$ -- a closed cube in $\mathbb R^N$ gives $\mathcal H_{\mathbb R^N}^{\psi_{N,1}}\bigl(f^{-1}(c)\bigr)<\infty$, $\psi_{N,1}(t):=\frac{t^{N-1}}{\max(|\ln t|^N, 1)}$, $t\in [0,\infty)$, for a.e. $c\in {\rm range}(f)$.\smallskip

\noindent (2) Let us formulate also the following problem concerning the structure of graphs of functions in the first Bernstein class:
\begin{Problem}
 Is it true that for each function $f\in \mathfrak{B}^*(C[0,1], \mathcal P(\mathbb C)|_{[0,1]})$ its graph $\Gamma_f([0,1])\subset\mathbb C^2$ is pluripolar?
 \end{Problem}
Note that the affirmative answer in this problem leads, by means of the argument of the proof of Corollary \ref{cor2.16}, to the one in Problem \ref{p2.19}.\smallskip

\noindent (3) Due to the classical result of Sadullaev \cite{Sa}, analogs of  Proposition \ref{clp} and Corollary \ref{cor2.16} are valid for $K$ being a compact non-pluripolar subset of a complex irreducible algebraic subvariety $M\subset\mathbb C^N$. In this case, one proves pluripolarity of $\Gamma_f(K)$ in $M\times\mathbb C$ and of $f^{-1}(c)$, $c\in {\rm range}(f)\setminus S_f$, in $M$.
}
\end{R}

\sect{Proofs of Results of Subsection~2.1}
\subsection{Proof of Proposition \ref{prop1}}
\begin{proof}
Since $\|S\otimes T\|\le \|S\|\cdot\|T\|$ and $S(V_n)\subset V_n'$ for all $n\ge n_0$, for an element $x\in A\hat\otimes_\alpha B$ and such $n$ we have
\[
\begin{array}{l}
\displaystyle
E_n\left((S\otimes T)(x)\right):=\inf_{h'\in V_n'\otimes B'}\|(S\otimes T)(x)-h'\|_{A'\hat\otimes_\alpha B'}\le \inf_{h\in V_n\otimes B}\|(S\otimes T)(x-h)\|_{A'\hat\otimes_\alpha B'}\\
\\
\displaystyle\le  \|S\|\cdot\|T\|\cdot\inf_{h\in V_n\otimes B}\|x-h\|_{A\hat\otimes_\alpha B}= \|S\|\cdot\|T\|\cdot E_n(x).
\end{array}
\]
This implies (because $\lim_{n\rightarrow\infty}\varphi(n)=0$)
\[
\varliminf_{n\rightarrow\infty}\bigl(E_n\bigl((S\otimes T)(x)\bigr)\bigr)^{\varphi(n)}\le\varliminf_{n\rightarrow\infty}\bigl(E_n(x)\bigr)^{\varphi(n)}
\]
which shows (see Definition \ref{def1}) that $S\otimes B$ sends $\mathfrak B_\varphi((A,V);B;\alpha)$ into $\mathfrak B_\varphi((A',V');B';\alpha)$.

In turn, for a weight function $\kappa$ we have 
\[
\sum_{n=n_0}^\infty \kappa(n)\ln \bigl(E_n\bigl((S\otimes T)(x)\bigr)\bigr)\le
\sum_{n=n_0}^\infty \kappa(n)\ln \bigl(E_n(x)\bigr)+\sum_{n=n_0}^\infty \kappa(n)\ln\bigl(\|S\|\cdot\|T\|\bigr).
\]
Since the second term on the right-hand side is finite,
$S\otimes B$ maps $\mathfrak B^*_\kappa((A,V);B;\alpha)$ into $\mathfrak B^*_\kappa((A',V');B';\alpha)$, as required.
\end{proof}

\subsection{Proof of Theorem \ref{prop2.1}}
\begin{proof}
(1) The proof exploits the idea of \cite{Ma}. Let $\{a_i\}_{i\in\mathbb N}\subset A$ be a countable dense subset. We assign to $a_i$ the real number 
\[
d_i:=i+\frac{1}{\varphi(s(a_i))},\quad {\rm where}\quad s(a_i)=\min\{k\, :\, a_i\in V_k\}.
\]
Since the union $\cup_i\,(a_i\otimes B)$ of proper closed subspaces $a_i\otimes B\subset A\hat\otimes_\alpha B$  is dense in $A\hat\otimes_\alpha B$,
set
\[
L:=\bigcap_{j=1}^\infty\bigcup_{i=j}^\infty \left\{x\in A\hat\otimes_\alpha B\, :\, {\rm dist}(x, a_i\otimes B)< 2^{-d_i}\right\}\bigcup\, (V\otimes B),
\] 
where ${\rm dist}(x, a_i\otimes B):=\inf_{b\in B}\|x-a_i\otimes b\|_{A\hat\otimes_\alpha B}$,
is comeager in $A\hat\otimes_\alpha B$. Note that if $x\in L\setminus \bigl(V\otimes B\bigr)$, then there exists a subsequence $\{a_{i_k}\}_{k\in\mathbb N}\subset\{a_i\}_{i\in\mathbb N}$ such that (cf. \eqref{ba})
\[
\bigl(E_{s(a_{i_k})}(x)\bigr)^{\varphi(s(a_{i_k}))}<2^{-i_k\cdot \varphi(s(a_{i_k}))-1}<\frac 12\quad {\rm and}\quad \lim_{k\rightarrow\infty}s(a_{i_k})=\infty.
\]
Thus, $\varliminf_{\,n\rightarrow\infty}(E_n(x))^{\varphi(n)}\le\frac 12$, that is, $x\in \mathfrak{B}_\varphi((A,V);B;\alpha)$. Hence, $L\subset \mathfrak{B}_\varphi((A,V);B;\alpha)$, so the latter set is comeager in $A\hat\otimes_\alpha B$.\medskip

(2) Given $x\in A\hat\otimes_\alpha B$ consider a scale function $\varphi_x:\mathbb N\rightarrow (0,\infty)$ defined by the formula: $\varphi_x(1)=1$ and for $n\ge 2$,
\begin{equation}\label{phix}
\varphi_x(n):=\left\{
\begin{array}{ccc}
1&{\rm if}&E_n(x)\ge \frac{1}{e}\medskip
\\
\displaystyle \frac{1}{|\ln E_n(x)|}&{\rm if}&0<E_n(x)<\frac{1}{e}\medskip
\\
\min\bigl(\varphi_x(n-1),\frac 1n\bigr)&{\rm if}&E_n(x)=0.
\end{array}
\right.
\end{equation}
Then for all sufficiently large $n$,
\[
\bigl(E_n(x)\bigr)^{\varphi_x(n)}\le\frac 1e\, ,
\]
that is, $x\in \mathfrak{B}_{\varphi_x}((A,V);B;\alpha)$; this proves that
$A\hat\otimes_\alpha B=\cup_\varphi \mathfrak{B}_\varphi((A,V);B;\alpha)$.\smallskip

Next, if $x\in A\hat\otimes_\alpha B\setminus (V\otimes B)$, then $x\in \mathfrak{B}_{\varphi_x}((A,V);B;\alpha)\setminus \mathfrak{B}_{\varphi_x^2}((A,V);B;\alpha)$ because
\[
\varliminf_{n\rightarrow\infty}\bigl(E_n(x)\bigr)^{\varphi_x^2(n)}=\lim_{n\rightarrow\infty}\left(\frac{1}{e}\right)^{\varphi_x(n)}=1.
\]
This implies that $V\otimes B=\cap_\varphi \mathfrak{B}_\varphi((A,V);B;\alpha)$. \smallskip

To prove the third statement, assume on the contrary that there exists a countable family of scale functions $\Phi$ such that $A\hat\otimes_\alpha B=\cup_{\varphi\in\Phi}\mathfrak{B}_\varphi((A,V);B;\alpha)$.
Let us show that this is impossible. Since $\mathfrak{B}_\varphi((A,V);B;\alpha)=\mathfrak{B}_{c\cdot\varphi}((A,V);B;\alpha)$ for each $c\in (0,\infty)$, without loss of generality we may assume that the family  $\Phi:=\{\varphi_i\}_{i\in\mathbb N}$ is such that $\varphi_i(n)\le\frac{1}{2^i}$ for all $i, n\in\mathbb N$. Let us define
\[
\tilde\varphi(n):=\sup_{i\in\mathbb N}\varphi_i(n),\quad n\in\mathbb N.
\]
\begin{Lm}\label{l3.2}
$\tilde\varphi$ is a scale function.
\end{Lm}
\begin{proof}
Clearly $\tilde\varphi$ is a bounded positive nonincreasing function. We must show that $\lim_{n\rightarrow\infty}\tilde\varphi(n)=0$. Assume, on the contrary, that there exists a subsequence $\{n_j\}_{j\in\mathbb N}$ such that $\lim_{j\rightarrow\infty}\tilde\varphi(n_j)>0$. By the definition of $\tilde\varphi$, for each $n\in\mathbb N$ there exists some $i_n\in\mathbb N$ such that
\[
\tilde\varphi(n)=\varphi_{i_n}(n).
\]
We must consider two cases.

(a) $\{i_{n_j}\}_{j\in\mathbb N}$ contains an infinite constant subfamily, i.e., $i_{n_j}$ coincides with some $i_0$ for infinitely many $j$. Then $\tilde\varphi(n_j)=\varphi_{i_0}(n_j)$ for infinitely many $j$. But the latter sequence tends to zero as $j\rightarrow\infty$. Thus, $\tilde\varphi$ tends to $0$ along a subsequence of $\{n_j\}$, a contradiction.\smallskip

(b)  $\{i_{n_j}\}_{j\in\mathbb N}$ does not contain an infinite constant subfamily, i.e., $\lim_{j\rightarrow\infty}i_{n_j}=\infty$.  Then we have
\[
\lim_{j\rightarrow\infty}\tilde\varphi(n_j)=\lim_{j\rightarrow\infty}\varphi_{i_{n_j}}(n_j)\le \lim_{j\rightarrow\infty}\frac{1}{2^{i_{n_j}}}=0,
\] 
a contradiction proving the lemma.
\end{proof}

Next, we consider the scale function $\bar{\varphi}:=\sqrt{\tilde\varphi}$.
Since $\varphi_i\le\tilde\varphi$ for all $i$ and $\tilde\varphi\le \bar{\varphi}$,  
\[
\bigcup_{i}\mathfrak{B}_{\varphi_i}((A,V);B;\alpha)\subset\mathfrak{B}_{\tilde\varphi}((A,V);B;\alpha)\subset \mathfrak{B}_{\bar{\varphi}}((A,V);B;\alpha).
\]
On the other hand, 
\[
\varlimsup_{n\rightarrow\infty}\frac{\bar{\varphi}(n)}{\tilde\varphi(n)}=\infty.
\]
Hence, due to property (4) (proved below), $ \mathfrak{B}_{\bar{\varphi}}((A,V);B;\alpha)\setminus \mathfrak{B}_{\tilde\varphi}((A,V);B;\alpha)\ne\emptyset$. This shows that
$A\hat\otimes_\alpha B\setminus\cup_{i}\mathfrak{B}_{\varphi_i}((A,V);B;\alpha)\ne\emptyset$, a contradiction completing the proof of the third statement of (2).\medskip

(3) Construction of $\varphi_{ub}$ is the same as of $\tilde\varphi$ above. In turn, if $\Phi=\{\varphi_i\}_{i\in\mathbb N}$, then 
we define
\[
\varphi_{lb}(n):=\min_{1\le i\le n}\varphi_i(n),\quad n\in\mathbb N.
\]
Clearly, $\varphi_{lb}$ is a scale function and for each $i\in\mathbb N$, $\varphi_{lb}(n)\le\varphi_i(n)$ for all $n\ge i$. This implies that $\mathfrak{B}_{\varphi_{lb}}((A,V);B;\alpha)\subset \mathfrak{B}_{\varphi_i}((A,V);B;\alpha)$ for all $i\in\mathbb N$, as required.\medskip

(4) We begin with some auxiliary results. Fix $b\in B$ of norm one and a linear functional $b^*\in B^*$ of norm one such that $b^*(b)=1$. Consider linear operators:
$I_b: A\rightarrow A\otimes B$, $I_b(a):=a\otimes b$ and ${\rm Id}\otimes b^*: A\otimes B\rightarrow A$, $({\rm Id}\otimes b^*)\bigl(\sum_{j}a_j\otimes b_j\bigr):=\sum_j b^*(b_j)\cdot a_j$.
\begin{Lm}\label{l3.1}
For a reasonable crossnorm $\alpha$ on $A\otimes B$ the operator $I_b$ is an isometric embedding and the operator ${\rm Id}\otimes b^*$ is bounded.
\end{Lm}
\begin{proof}
By the definition of a reasonable crossnorm,
\[
\alpha(I_b(a))=\alpha(a\otimes b)=\|a\|_A\cdot\|b\|_B=\|a\|_A.
\]
This proves the first statement.
 
Next, for a linear functional $a^*\in A^*$ we clearly have $a^*\circ ({\rm Id}\otimes b^*)=a^*\otimes b^*$. By the definition of a reasonable crossnorm, $\alpha'(a^*\otimes b^*)=\|a^*\|_{A^*}\cdot\|b^*\|_{B^*}$, see, e.g., \cite{R}; here $\alpha'$ is the dual norm of $\alpha$. Hence, for all $x\in A\otimes B$,
\[
\begin{array}{l}
\displaystyle
\|({\rm Id}\otimes b^*)(x)\|_A=\sup_{a^*\in A^*\, ,\,\|a^*\|_{A^*}=1}|a^*\bigl(({\rm Id}\otimes b^*)(x)\bigr)|=\sup_{a^*\in A^*\, ,\,\|a^*\|_{A^*}=1}|(a^*\otimes b^*)(x)|\\
\\
\displaystyle \le \sup_{a^*\in A^*\, ,\,\|a^*\|_{A^*}=1}\alpha'(a^*\otimes b^*)\cdot\|x\|_{A\otimes B}=\|b\|_{B*}\cdot\|x\|_{A\otimes B}=\|x\|_{A\otimes B},
\end{array}
\]
as required.
\end{proof}
The lemma implies that operator ${\rm Id}\otimes b^*$ is extended by continuity to a bounded linear operator $A\hat\otimes_\alpha B\rightarrow A$ of norm $\le 1$ (denoted by the same symbol). 
\begin{Lm}\label{l3.3}
For each $a\in A$, 
\[
E_n(a):=\inf_{h\in V_n}\|a-h\|_A=E_n(I_b(a)):=\inf_{h\in V_n\otimes B}\alpha(I_b(a)-h).
\]
\end{Lm}
\begin{proof}
Since $I_b$ is an isometric embedding, $E_n(I_b(a))\le \|I_b\|\cdot E_n(a)=E_n(a)$. Conversely, since $({\rm Id}\otimes b^*)\circ I_b={\rm Id}$ and $\|{\rm Id}\otimes b^*\|\le 1$, $E_n(a)\le \|{\rm Id}\otimes b^*\|\cdot E_n(I_b(a))=E_n(I_b(a))$.
\end{proof}
We are ready to prove property (4). By Definition \ref{def1}, if $\varphi'\le c\cdot\varphi$ for some $c\in (0,\infty)$, then $\mathfrak{B}_{\varphi'}((A,V);B;\alpha)\subset \mathfrak{B}_{\varphi}((A,V);B;\alpha)$. Now assume that the latter inclusion holds but there exists a subsequence $\{n_i\}_{i\in\mathbb N}\subset\mathbb N$ such that 
\begin{equation}\label{e3.1}
\lim_{i\rightarrow\infty}\frac{\varphi'(n_i)}{\varphi(n_i)}=\infty.
\end{equation}
Let us define a sequence $\{c_n\}_{n\in\mathbb N}\subset (0,\infty)$ by the formula 
\[
c_n:=e^{-\frac{1}{\varphi'(n_i)}},\quad n_i\le n<n_{i+1},\ n\in\mathbb N.
\]
Clearly, $c_{n+1}\le c_n$  for all $n$, and $\lim_{n\rightarrow\infty} c_n=0$. As follows from the Bernstein theorem (adapted to the case of Banach spaces by A. Timan, see \cite[Sect.\,2.5]{T} and also \cite[Th.\,5.1]{P2}), there exists an element $a\in A$ such that $E_n(a)=c_n$ for all $n$. Then, due to Lemma \ref{l3.3}, the element $\tilde a:=I_b(a)$ satisfies $E_n(\tilde a)=c_n$ for all $n$. Note that $\tilde a\in \mathfrak{B}_{\varphi'}((A,V);B;\alpha)$. Indeed,
\[
\varliminf_{i\rightarrow\infty}(E_{n_i}(\tilde a))^{\varphi'(n_i)}=\varliminf_{i\rightarrow\infty}\left(e^{-\frac{1}{\varphi'(n_i)}}\right)^{\varphi'(n_i)}=\frac{1}{e}.
\]
Let us show that $\tilde a\not\in \mathfrak{B}_{\varphi}((A,V);B;\alpha)$. In fact, observe that 
\[
c_n^{\,\varphi(n)}=e^{-\frac{\varphi(n)}{\varphi'(n_i)}}\geq e^{-\frac{\varphi(n_i)}{\varphi'(n_i)}}\quad {\rm for}\ n_i\le n<n_{i+1}.
\]
Hence, due to \eqref{e3.1},
\[
\varliminf_{n\rightarrow\infty}c_n^{\,\varphi(n)}=1.
\]
This implies
\[
\varliminf_{n\rightarrow\infty}(E_n(\tilde a))^{\varphi(n)}=\varliminf_{n\rightarrow\infty}c_n^{\,\varphi(n)}=1.
\]
Thus, $\tilde a\not\in \mathfrak{B}_{\varphi}((A,V);B;\alpha)$ which contradicts our assumption and shows that \eqref{e3.1} is false. Therefore
\[
\varlimsup_{n\rightarrow\infty}\frac{\varphi'(n)}{\varphi(n)}<\infty,
\]
as required.
\end{proof}

\subsection{Proof of Corollary \ref{lat}}
\begin{proof}
Let $Q=((A,V);B;\alpha)\in\mathscr C$.
By definition, $\mathfrak B_{\{\varphi_1\}}(Q)\vee \mathfrak B_{\{\varphi_2\}}(Q)$ is the minimal element of $(\mathscr{B}(Q),\subseteq)$ containing both of these sets, in particular,
\begin{equation}\label{e3.2}
\mathfrak B_{\{\varphi_1\}}(Q)\cup \mathfrak B_{\{\varphi_2\}}(Q)\subseteq \mathfrak B_{\{\varphi_1\}}(Q)\vee \mathfrak B_{\{\varphi_2\}}(Q).
\end{equation}
Since $\varphi_i\le \max(\varphi_1,\varphi_2)$, $i=1,2$, \begin{equation}\label{e3.3}
\mathfrak B_{\{\varphi_1\}}(Q)\vee \mathfrak B_{\{\varphi_2\}}(Q)\subseteq \mathfrak B_{\{\max(\varphi_1,\varphi_2)\}}(Q).
\end{equation}

On the other hand, let $x\in \mathfrak B_{\{\max(\varphi_1,\varphi_2)\}}(Q)$. Then there exists a subsequence $\{n_k\}_{k\in\mathbb N}\subset\mathbb N$ such that for some $\rho\in (0,1)$ and all $k$,
\[
(E_{n_k}(x))^{\max(\varphi_1(n_k),\varphi_2(n_k))}\le \rho .
\]
Passing to a subsequence, if necessary,  we may assume that
$\max(\varphi_1(n_k),\varphi_2(n_k))=\varphi_i(n_k)$ for all $k$, where $i\in\{1,2\}$. Then the previous inequality implies that
$x\in \mathfrak B_{\{\varphi_i\}}(Q)$. This shows that
\begin{equation}\label{e3.4}
\mathfrak B_{\{\max(\varphi_1,\varphi_2)\}}(Q)\subseteq \mathfrak B_{\{\varphi_1\}}(Q)\cup \mathfrak B_{\{\varphi_2\}}(Q).
\end{equation}

Combining \eqref{e3.2}--\eqref{e3.4} we get the first statement of the corollary.\smallskip

Further, by definition, $\mathfrak B_{\{\varphi_1\}}(Q)\wedge \mathfrak B_{\{\varphi_2\}}(Q)$ is the maximal element of $(\mathscr{B}(Q),\subseteq)$ contained in both of these sets. In particular,
\[
\mathfrak B_{\{\varphi_1\}}(Q)\wedge \mathfrak B_{\{\varphi_2\}}(Q)\subseteq \mathfrak B_{\{\varphi_1\}}(Q)\cap \mathfrak B_{\{\varphi_2\}}(Q).
\]
Since $\varphi_i\ge \min(\varphi_1,\varphi_2)$, $i=1,2$, Theorem \ref{prop2.1}\,(4) implies that
\begin{equation}\label{e3.5}
\mathfrak B_{\{\min(\varphi_1,\varphi_2)\}}(Q)\subseteq \mathfrak B_{\{\varphi_1\}}(Q)\wedge \mathfrak B_{\{\varphi_2\}}(Q).\end{equation}
Further, since each $\mathfrak B_{\{\varphi\}}(Q)\subseteq \mathfrak B_{\{\varphi_1\}}(Q)\wedge \mathfrak B_{\{\varphi_2\}}(Q)$ belongs to $\mathfrak B_{\{\varphi_1\}}(Q)\cap \mathfrak B_{\{\varphi_2\}}(Q)$, Theorem \ref{prop2.1}\,(4) implies that there exists a constant $c>0$ such that $\varphi\le c\cdot\varphi_i$, $i=1,2$. That is,
$\varphi\le c\cdot \min(\varphi_1,\varphi_2)$ which yields
\[
\mathfrak B_{\{\varphi\}}(Q)\subseteq \mathfrak B_{\{\min(\varphi_1,\varphi_2)\}}(Q).
\]
This and \eqref{e3.5} show that
\[
\mathfrak B_{\{\min(\varphi_1,\varphi_2)\}}(Q)= \mathfrak B_{\{\varphi_1\}}(Q)\wedge \mathfrak B_{\{\varphi_2\}}(Q),
\]
as required.
\end{proof}

\subsection{Proof of Theorem \ref{prop2.2}}
\begin{proof}
(1) Let $Q=((A,V);B;\alpha)\in\mathscr C$. First, we show that
\begin{equation}\label{embed}
\mathfrak{B}_{\Sigma(\kappa)}(Q)\subset \mathfrak{B}^*_{\kappa}(Q).
\end{equation}
To avoid abuse of notation, we set 
\begin{equation}\label{abuse}
\varphi:=\Sigma(\kappa),\quad \varphi(n)=\Sigma(\kappa)(n):=\sum_{i=n}^\infty\kappa(i),\quad n\in\mathbb N.
\end{equation}
Now,
for $x\in \mathfrak{B}_{\varphi}(Q)\setminus (V\otimes B)$, let $\{n_k\}_{k\in\mathbb N}\subset\mathbb N$ be a subsequence such that
\[
\frac{\varphi(n_{k+1})}{\varphi(n_k)}\le\frac 12\quad {\rm and}\quad
(E_{n_k}(x))^{\varphi(n_k)}<\rho<1\quad {\rm for\ all}\quad k\in\mathbb N.
\]
By the definition of the best approximation, $0<E_n(x)\le E_{n_k}(x)$ for all $n_{k}\le n< n_{k+1}$. From here and the previous inequalities we obtain (since $\ln\rho<0$)
\[
\begin{array}{l}
\displaystyle
\sum_{n=n_1}^\infty \kappa(n)\cdot\ln (E_n(x))=\sum_{k=1}^\infty\sum_{n=n_k}^{n_{k+1}-1}\kappa(n)\cdot \ln(E_n(x))\le\sum_{k=1}^\infty\sum_{n=n_k}^{n_{k+1}-1}\kappa(n)\cdot\ln (E_{n_k}(x))\\
\\
\displaystyle \le \sum_{k=1}^\infty \frac{\ln\rho\cdot (\varphi(n_k)-\varphi(n_{k+1}) )}{\varphi(n_k)}\le \sum_{k=1}^\infty \frac{\ln\rho}{2}=-\infty.
\end{array}
\]
This implies that $x\in \mathfrak{B}_{\kappa}^*(Q)$, see \eqref{beurling}, and proves \eqref{embed}.\medskip

Next, we show that $\mathfrak{B}_{\Sigma(\kappa)}(Q)\ne \mathfrak{B}^*_{\kappa}(Q)$. The proof is based on
\begin{Lm}\label{lem3.4}
There exists an increasing to $\infty$ function $h:\mathbb N\rightarrow (0,\infty)$ such that function $h\cdot\varphi$  decreases to $0$, and
\[
\sum_{n=1}^\infty\frac{\kappa(n)}{h(n)\cdot\varphi(n)}=\infty .
\]
\end{Lm}
\begin{proof}

Without loss of generality we may assume that $\varphi(n)\le 1$ for all $n$. First, observe that the positive function $\frac{1}{\varphi}$ increases to $\infty$ and satisfies
\begin{equation}\label{equ3.3}
\sum_{n=1}^\infty\frac{\kappa(n)}{\varphi(n)}=\infty.
\end{equation}
Using this, we construct by induction an increasing subsequence 
$\{N_k\}_{k\in\mathbb Z_+}\subset\mathbb Z_+$ and a decreasing to $0$ sequence $\{\alpha_k\}_{k\in\mathbb Z_+}\subset (0,\infty)$ such that
\begin{itemize}
\item[(i)]
\[
\frac{\varphi^{1-\alpha_{k-1}}(N_{k-1}+1)}{\varphi^{\alpha_{k}}(N_{k-1}+1)\cdot\varphi^{1-\alpha_{k-1}}(N_{k-1})}<1,\qquad  k\ge 2;
\]
\item[(ii)]
\[
\sum_{n=N_{k-1}+1}^{N_{k}}\frac{\left(\prod_{i=0}^{k-1}\varphi^{\alpha_{i}}(N_{i}+1)\right)\cdot\kappa(n)}{\varphi^{1-\alpha_{k}}(n)}>1,\qquad  k\in\mathbb N;
\]
\item[(iii)]
\[
\frac{\varphi^{1-\alpha_k}(N_k)}{\prod_{i=0}^{k-1}\varphi^{\alpha_{i}}(N_{i}+1)}\le \frac 1k,\qquad  k\in\mathbb N;
\]
\item[(iv)]
\[ 
\left(\prod_{i=0}^{k-1}\varphi^{-\alpha_{i}}(N_{i}+1)\right)\cdot \varphi^{-\alpha_k}(N_k)\ge k,\qquad k\in\mathbb N.
\]
\end{itemize}
We set $N_0=0$, $\alpha_0=1$ and assume that the required $N_1,\dots, N_{k-1}$ and $\alpha_1,\dots, \alpha_{k-1}$ are chosen. Let us define $N_{k}$ and $\alpha_{k}$. To this end, due to \eqref{equ3.3} there exists some integer $N_{k}'>N_{k-1}$ such that
\[
\sum_{n=N_{k-1}+1}^{N_{k}'}\frac{\left(\prod_{i=0}^{k-1}\varphi^{\alpha_{i}}(N_{i}+1)\right)\cdot\kappa(n)}{\varphi(n)}>1.
\]
Thus we can choose a positive number $\alpha_{k}<\alpha_{k-1}$ such that
\[
\sum_{n=N_{k}+1}^{N_{k}'}\frac{\left(\prod_{i=0}^{k-1}\varphi^{\alpha_{i}}(N_{i}+1)\right)\cdot\kappa(n)}{\varphi^{1-\alpha_{k}}(n)}>1,
\]
and for $k\ge 2$ (because $\varphi$ is decreasing),
\[
 \frac{\varphi^{1-\alpha_{k-1}}(N_{k-1}+1)}{\varphi^{-\alpha_{k}}(N_{k-1}+1)\cdot \varphi^{1-\alpha_{k-1}}(N_{k-1})}<1.
\]
Further, since $\varphi^{1-\alpha_k}$ is decreasing to $0$ and $\varphi^{-\alpha_k}$ is increasing to $\infty$, we can choose an integer $N_{k}>N_{k}'$ such that 
\[
\frac{\varphi^{1-\alpha_k}(N_k)}{\prod_{i=0}^{k-1}\varphi^{\alpha_{i}}(N_{i}+1)}\le \frac 1k\quad {\rm and}\quad \left(\prod_{i=0}^{k-1}\varphi^{-\alpha_{i}}(N_{i}+1)\right)\cdot \varphi^{-\alpha_k}(N_k)\ge k,
\]
as required.

Now, let us define 
\begin{equation}\label{equ3.5}
h(n):=\left(\prod_{i=0}^{k-1}\varphi^{-\alpha_{i}}(N_{i}+1)\right)\cdot \varphi^{-\alpha_{k}}(n), \quad N_{k-1}+1\le n\le N_{k},\quad k\in\mathbb N.
\end{equation}
So $\{h(n)\}$ is increasing for $N_{k-1}+1\le n\le N_{k}$ and, because $\varphi$ is decreasing and $\le 1$,
\[
\frac{h(N_k+1)}{h(N_k)}=\frac{\varphi^{-\alpha_{k}}(N_{k}+1)\cdot\varphi^{-\alpha_{k+1}}(N_k+1)}{\varphi^{-\alpha_{k}}(N_{k})}\ge 1.
\]
Thus $h$ is increasing and since by (iv) $h(N_k)\ge k$, it is unbounded.

Next,
\[
h(n)\cdot\varphi(n)=\left(\prod_{i=0}^{k-1}\varphi^{-\alpha_{i}}(N_{i}+1)\right)\cdot \varphi^{1-\alpha_k}(n),\quad N_{k-1}+1\le n\le N_{k}.
\]
So $h\cdot\varphi$ is decreasing on each such an interval. Also, by (i),
\[
\frac{h(N_k+1)\cdot\varphi(N_k+1)}{h(N_k)\cdot\varphi(N_k)}=
\frac{\varphi^{1-\alpha_{k}}(N_{k}+1)\cdot\varphi^{-\alpha_{k+1}}(N_k+1)}{\varphi^{1-\alpha_k}(N_k)}<1,
\]
so $h\cdot\varphi$ is decreasing and since by (iii)
\[
h(N_k)\cdot\varphi(N_k)=\left(\prod_{i=0}^{k-1}\varphi^{-\alpha_{i}}(N_{i}+1)\right)\cdot \varphi^{1-\alpha_k}(N_k)\le\frac{1}{k},
\]
$\lim_{n\rightarrow\infty}h(n)\cdot\varphi(n)=0$.

Finally, by (ii),
\[
\sum_{j=1}^\infty\frac{\kappa(j)}{h(j)\cdot\varphi(j)}=\sum_{k=0}^\infty\sum_{n=N_{k-1}+1}^{N_{k}}\frac{\left(\prod_{i=0}^{k-1}\varphi^{\alpha_{i}}(N_{i}+1)\right)\cdot\kappa(n)}{\varphi^{1-\alpha_{k}}(n)}>\sum_{k=1}^\infty 1=\infty.
\]
This completes the proof of the lemma.
\end{proof}

Further, as in the proof of Theorem \ref{prop2.1}\,(4) (see the version of the Bernstein theorem), there exists an element $x\in A\hat\otimes_\alpha B$ such that
\[
E_n(x)=e^{-\frac{1}{h(n)\cdot\varphi(n)}},\quad n\in\mathbb N.
\]
(Here we used that $h\cdot\varphi$ decreases to $0$.) 
We have
\[
\sum_{k=1}^\infty \kappa(n)\cdot\ln (E_n(x))=-\sum_{k=1}^\infty \frac{\kappa(n)}{h(n)\cdot\varphi(n)}=-\infty,
\]
that is, $x\in\mathfrak{B}^*_{\kappa}(Q)$. On the other hand, by Lemma \ref{lem3.4},
\[
\varliminf_{n\rightarrow\infty}(E_n(x))^{\varphi(n)}=
\varliminf_{n\rightarrow\infty}e^{-\frac{1}{h(n)}}=1,
\]
that is, $x\not\in\mathfrak{B}_{\Sigma(\kappa)}(Q)$. Thus, we have proved that
\[
\mathfrak{B}_{\Sigma(\kappa)}(Q)\subsetneq \mathfrak{B}^*_{\kappa}(Q).
\]

To complete the proof of part (1) of the theorem, we need to show that
$\mathfrak{B}_{\Sigma(\kappa)}(Q)$ is maximal among all subsets in $\mathscr{B}(Q)$ containing in $\mathfrak{B}^*_{\kappa}(Q)$.

Assume, on the contrary, that there exists  a subset $\mathfrak{B}_{\varphi'}(Q)\supsetneq \mathfrak{B}_{\Sigma(\kappa)}(Q)$ of
$\mathfrak{B}^*_{\kappa}(Q)$. Then due to Theorem \ref{prop2.1}\,(4), there exists an increasing subsequence $\{n_k\}_{k\in\mathbb N}\subset\mathbb N$ such that
\[
\lim_{k\rightarrow\infty}\frac{\varphi(n_k)}{\varphi'(n_k)}=0,
\]
where $\varphi(n):=\sum_{i=n}^\infty \kappa(i)$, $n\in\mathbb N$.

Passing to another subsequence, if necessary, we may assume that 
\[
\sum_{k=1}^\infty \frac{\varphi(n_k)}{\varphi'(n_k)}<\infty.
\]
We set 
\[
\tilde\varphi'(n):=\varphi'(n_k),\quad {\rm for}\quad n_{k}\le n< n_{k+1},\quad k,\, n\in\mathbb N.
\]
Since $\tilde\varphi'$ is nonincreasing and tends to $0$, by the version of the Bernstein theorem (see the proof of Theorem \ref{prop2.1}\,(4) above) there exists an element $x\in A\hat\otimes_\alpha B$ such that
\[
E_n(x)=e^{-\frac{1}{\tilde\varphi'(n)}},\quad n\in\mathbb N.
\]
We have
\[
\varliminf_{n\rightarrow\infty} (E_n(x))^{\varphi'(n)}=\varliminf_{n\rightarrow\infty}e^{-\frac{\varphi'(n)}{\tilde\varphi'(n)}}=e^{-1}.
\]
Hence, $x\in \mathfrak{B}_{\varphi'}(Q)$. On the other hand,
\[
\begin{array}{l}
\displaystyle
\sum_{n=n_1}^\infty \kappa(n)\cdot\ln (E_n(x))=-\sum_{k=1}^\infty\sum_{n=n_k}^{n_{k+1}-1}\frac{\kappa(n)}{\tilde\varphi'(n)}=-\sum_{k=1}^\infty\sum_{n=n_k}^{n_{k+1}-1}\frac{\kappa(n)}{\varphi'(n_k)}\\
\\
\displaystyle =- \sum_{k=1}^\infty \frac{\varphi(n_k)-\varphi(n_{k+1}) }{\varphi'(n_k)}\ge -\sum_{k=1}^\infty \frac{\varphi(n_k)}{\varphi'(n_k)}>-\infty.
\end{array}
\]
Thus $x\not\in \mathfrak{B}^*_{\kappa}(Q)$, a contradiction proving the claim.

The proof of part (1) of the theorem is complete.\medskip

(2) Assume that
$\mathfrak B_{\Sigma(\kappa_1)}(Q)=\mathfrak B_{\Sigma(\kappa_2)}(Q)$.
We set for brevity $\varphi_i:=\Sigma(\kappa_i)$, $i=1,2$. Then, due to Theorem \ref{prop2.1}\,(4), for some $c_1,c_2\in (0,\infty)$,
\begin{equation}\label{equivalence}
c_1\cdot\varphi_2\le\varphi_1\le c_2\cdot\varphi_2.
\end{equation}
Let $x\in \mathfrak B^*_{\kappa_1}(Q)\setminus V\otimes B$, that is,
\[
\sum_{n=1}^\infty \kappa_1(n)\cdot\ln(E_n(x))=-\infty.
\]
By $S_{k,1}$ we denote the $k$th partial sum of the series. Since $\kappa_1(n)=\varphi_1(n)-\varphi_1(n+1)$ for all $n$, 
\[
\begin{array}{l}
\displaystyle
S_{k,1}=\sum_{n=1}^k \kappa_1(n)\cdot\ln(E_n(x))=\sum_{n=1}^k (\varphi_1(n)-\varphi_1(n+1))\cdot\ln(E_n(x))\\
\\
\displaystyle =\sum_{n=2}^{k}\varphi_1(n)\cdot\ln\left(\frac{E_{n}(x)}{E_{n-1}(x)}\right)+\varphi_1(1)\cdot\ln(E_1(x))-\varphi_1(k+1)\cdot\ln(E_k(x)).
\end{array}
\]
We consider two cases.

(a)
\[
\varliminf_{k\rightarrow\infty}\varphi_1(k+1)\cdot\ln(E_k(x))=-\infty.
\]
Since $\varphi_1(k+1)\le\varphi_1(k)$, this implies that
\[
0=\varliminf_{k\rightarrow\infty}(E_k(x))^{\varphi_1(k+1)}\ge \varliminf_{k\rightarrow\infty}(E_k(x))^{\varphi_1(k)}.
\]
That is, $x\in\mathfrak B_{\varphi_1}(Q)=\mathfrak B_{\varphi_2}(Q)$.
Hence, due to part (1) of the theorem, $x\in \mathfrak{B}^*_{\kappa_2}(Q)$.

(b)
\[
\varliminf_{k\rightarrow\infty}\varphi_1(k+1)\cdot\ln(E_k(x))>-\infty.
\]
Since $\lim_{k\rightarrow\infty} S_{k,1}=-\infty$, the latter condition, equivalence \eqref{equivalence} and the fact that $\{E_n\}_{n\in\mathbb N}$ is nonincreasing imply that
\[
-\infty=\lim_{k\rightarrow\infty}\left(\sum_{n=2}^{k}\varphi_1(n)\cdot\ln\left(\frac{E_{n}(x)}{E_{n-1}(x)}\right) \right)\ge \lim_{k\rightarrow\infty}\left(\sum_{n=2}^{k}c_2\cdot \varphi_2(n)\cdot\ln\left(\frac{E_{n}(x)}{E_{n-1}(x)}\right) \right)
\]
and
\[
\varliminf_{k\rightarrow\infty}c_1\cdot \varphi_2(k+1)\cdot\ln(E_k(x))>-\infty.
\]
Thus for
\[
\begin{array}{l}
\displaystyle
S_{k,2}=\sum_{n=1}^k \kappa_2(n)\cdot\ln(E_n(x))=\sum_{n=1}^k (\varphi_2(n)-\varphi_2(n+1))\cdot\ln(E_n(x))\\
\\
\displaystyle =\sum_{n=2}^{k}\varphi_2(n)\cdot\ln\left(\frac{E_{n}(x)}{E_{n-1}(x)}\right)+\varphi_2(1)\cdot\ln(E_1(x))-\varphi_2(k+1)\cdot\ln(E_k(x)),
\end{array}
\]
we get
\[
\sum_{n=1}^\infty \kappa_2(n)\cdot\ln(E_n(x))=\lim_{k\rightarrow\infty}S_{k,2}=-\infty,
\]
that is, $x\in \mathfrak B_{\kappa_2}^*(Q)$. This shows that $\mathfrak B_{\kappa_1}^*(Q)\subset \mathfrak B_{\kappa_2}^*(Q)$. By the symmetry of the above arguments, $\mathfrak B_{\kappa_2}^*(Q)\subset \mathfrak B_{\kappa_1}^*(Q)$, i.e., these two sets coincide.

Conversely, suppose $\mathfrak B_{\kappa_1}^*(Q)=\mathfrak B_{\kappa_2}^*(Q)=:U$. Then according to part (1) of the theorem
$\mathfrak B_{\Sigma(\kappa_i)}(Q)$, $i=1,2$, are maximal elements of $\mathscr B(Q)$ containing in $U$. Since by Corollary \ref{lat} $(\mathscr B(Q),\subseteq)$ is a lattice, these maximal elements must coincide, i.e. $\mathfrak B_{\Sigma(\kappa_1)}(Q)=\mathfrak B_{\Sigma(\kappa_2)}(Q)$. This completes the proof of part (2) of the theorem.\medskip

(3) Since each scale function $\varphi$ is equivalent to a decreasing one, say, $\tilde\varphi$,
for the weight function $\kappa(n):=\tilde\varphi(n)-\tilde\varphi(n+1)$, $n\in\mathbb N$, by \eqref{equ2.1} we have $\{\Sigma(\kappa)\}=\{\tilde\varphi\}=\{\varphi\}$. This shows that $\mathfrak B_{\Sigma(\kappa)}(Q)=\mathfrak B_\varphi(Q)$ and proves part (3).

\medskip

(4) Let $\varphi\in\Phi_\kappa$. We will show that for each $f\in \mathfrak{B}_\kappa^*(Q)$,
\[
\varliminf_{n\rightarrow\infty}\bigl(E_n(f)\bigr)^{\varphi(n)}=0.
\] 
This will imply the implication 
\[
\mathfrak{B}_\kappa^*(Q)\subset \bigcap_{\varphi\in\Phi_\kappa}\mathfrak{B}_\varphi(Q).
\]

Assume, on the contrary, that for some $x\in \mathfrak{B}_\kappa^*(Q)$ there exists a $\rho\in (0,1)$ such that
\[
E_n(x)\ge\rho^{\frac{1}{\varphi(n)}}\quad {\rm for\ all}\quad n\in\mathbb N.
\]
Then due to \eqref{beurling}
\[
-\infty=\sum_{n=1}^\infty \kappa(n)\cdot\ln (E_n(x))\ge \sum_{n=1}^\infty \frac{\kappa(n)\cdot \ln\rho}{\varphi(n)}>-\infty,
\]
a contradiction proving the required implication.

To prove the opposite implication, assume, on the contrary, that there exists
\[
x\in  \left(\bigcap_{\varphi\in\Phi_\kappa}\mathfrak{B}_\varphi(Q)\right)\setminus
\mathfrak{B}_\kappa^*(Q).
\]
Then, due to \eqref{beurling},
\[
\sum_{n=1}^\infty \kappa(n)\cdot\ln (E_n(x))>-\infty.
\]
This shows that the scale function $\varphi_x$ defined in the proof of part (2) of Theorem \ref{prop2.1} (see \eqref{phix}) belongs to the class $\Phi_\kappa$. A standard argument shows that there exists a decreasing to $0$ function $\mu:\mathbb N\rightarrow (0,\infty)$ such that $\mu\cdot\varphi_x$ still belongs to $\Phi_\kappa$. Thus we must have $x\in\mathfrak{B}_{\mu\cdot\varphi_x}(Q)$, that is,
\[
\varliminf_{n\rightarrow\infty}(E_n(x))^{\mu(n)\cdot\varphi(n)}< 1.
\]
However, by the definition of $\varphi_x$ (because $x\not\in V\otimes B$),
\[
\varliminf_{n\rightarrow\infty}(E_n(x))^{\mu(n)\cdot\varphi(n)}=\lim_{n\rightarrow\infty}\left(\frac{1}{e}\right)^{\mu(n)}=1,
\]
a contradiction showing that
\[
\mathfrak{B}_\kappa^*(Q)=\bigcap_{\varphi\in\Phi_\kappa}\mathfrak{B}_\varphi(Q).
\]

The proof of the theorem is complete.
\end{proof}

\subsection{Proof of Corollary \ref{cor2.4}}
\begin{proof}
(1) Since $\mathfrak{B}^*_{\kappa}(Q)\supset\mathfrak{B}_{\Sigma(\kappa)}(Q)$, $Q=((A,V);B;\alpha)\in\mathscr C$, and the latter set is comeager in $A\hat\otimes_\alpha B$, the former one is comeager in $A\hat\otimes_\alpha B$ as well.
\smallskip

\noindent (2) Since, due to Theorem \ref{prop2.2}\,(3), for each scale function $\varphi$ there exists a weight function $\kappa$ such that $\mathfrak{B}_{\Sigma(\kappa)}(Q)=\mathfrak{B}_{\varphi}(Q)$, Theorem \ref{prop2.1}\,(2) implies that
\[
A\hat\otimes_\alpha B\supseteq\bigcup\mathfrak B^*_\kappa(Q)\supseteq\bigcup\mathfrak B_\varphi(Q)=A\hat\otimes_\alpha B,
\]
where the unions are taken over all weight and scale functions. Thus, the union of all classes $\mathfrak B^*_\kappa(Q)$ coincides with $A\hat\otimes_\alpha B$.

Further, for a scale function $\varphi$ consider the weight function
\[
\kappa_{\varphi}(n):=\frac{\varphi(n)}{n^2},\quad n\in\mathbb N.
\]
Then
\[
\sum_{n=1}^\infty \frac{\kappa_\varphi(n)}{\varphi(n)}=\sum_{n=1}^\infty\frac{1}{n^2}<\infty.
\]
Hence, $\varphi\in \Phi_{\kappa_\varphi}$, see Theorem \ref{prop2.2}\,(4). Due to this theorem, 
\[
\mathfrak B^*_{\kappa_\varphi}(Q)\subset \mathfrak B_\varphi(Q).
\]
This and Theorem \ref{prop2.1}\,(2) imply that
\[
V\otimes B\subseteq\bigcap_{w\in\mathcal W}\mathfrak B^*_w(Q)\subseteq\bigcap_{s\in\mathcal S}\mathfrak B_s(Q)=V\otimes B.
\]
Thus, the intersection of all classes $\mathfrak B^*_\kappa(Q)$ coincides with $V\otimes B$.

Finally, $A\hat\otimes_\alpha B\setminus\cup_{\kappa\in K}\mathfrak{B}^*_\kappa(Q)\ne\emptyset$ for any countable subset $K$ of weight functions because due to Theorem \ref{prop2.2}\,(4) each $\mathfrak{B}^*_\kappa(Q)$ is a subset of some
$\mathfrak{B}_\varphi(Q)$ and a similar property is valid for the second Bernstein classes by Theorem \ref{prop2.1}\,(2).
\smallskip

\noindent (3) Let $K$ be a countable subset of weight functions. Then due to Theorem \ref{prop2.1}\,(3) there exists some scale function $\varphi_{lb}$
\[
\bigcap_{\kappa\in K}\mathfrak B^*_\kappa(Q)\supseteq \bigcap_{\kappa\in K}\mathfrak B_{\Sigma(\kappa)}(Q)\supseteq\mathfrak B_{\varphi_{lb}}(Q).
\]
Hence, see the argument in the proof of part (2) above,
\[
\bigcap_{\kappa\in K}\mathfrak B^*_\kappa(Q)\supseteq \mathfrak B^*_{\kappa_{\varphi_{lb}}}(Q).
\]
Similarly, for each $\kappa\in K$ choose some $\varphi_\kappa\in\Phi_\kappa$ (see Theorem \ref{prop2.2}\,(4)). Then due to Theorem \ref{prop2.1}\,(3) there exists some scale function $\varphi_{ub}$ such that
\[
\bigcup_{\kappa\in K}\mathfrak B^*_\kappa(Q)\subseteq \bigcup_{\kappa\in K}\mathfrak B_{\varphi_\kappa}(Q)\subseteq \mathfrak B_{\varphi_{ub}}(Q).
\]
Also, due to Theorem \ref{prop2.2}\,(3), there exists some $\kappa_{ub}$ such that 
\[
\mathfrak B_{\varphi_{ub}}(Q)\subseteq \mathfrak B^*_{\kappa_{ub}}(Q).
\]
Hence,
\[
\bigcup_{\kappa\in K}\mathfrak B^*_\kappa(Q)\subseteq  \mathfrak B^*_{\kappa_{ub}}(Q).
\]

This completes the proof of analogs of properties (1)--(3) of Theorem \ref{prop2.1} for the first Bernstein classes.
\end{proof}

\subsection{Proof of Proposition \ref{prop2.5}}
\begin{proof}
As in the proof of Corollary \ref{lat} we get
\begin{equation}\label{e3.6}
\begin{array}{l}
\displaystyle
\mathfrak B^*_{\kappa_1}(Q)\cup \mathfrak B^*_{\kappa_2}(Q)\subseteq \mathfrak B^*_{\kappa_1}(Q)\vee \mathfrak B^*_{\kappa_2}(Q);\\
\\
\displaystyle
 \mathfrak B^*_{\kappa_1}(Q)\vee \mathfrak B^*_{\kappa_2}(Q)\subseteq \mathfrak B^*_{\max(\kappa_1,\kappa_2)}(Q).
\end{array}
\end{equation}
Also, if $x\not\in \mathfrak B^*_{\kappa_1}(Q)\cup \mathfrak B^*_{\kappa_2}(Q)$, then for $i=1,2$,
\[
\sum_{n=1}^\infty \kappa_i(n)\cdot\ln(E_n(x))>-\infty.
\]
Adding these inequalities term by term (which is possible because the series converge absolutely) we obtain
\[
\sum_{n=1}^\infty (\kappa_1(n)+\kappa_2(n))\cdot\ln(E_n(x))>-\infty.
\]
Hence, $x\not\in \mathfrak B^*_{\kappa_1+\kappa_2}(Q)=\mathfrak B^*_{\max(\kappa_1,\kappa_2)}(Q)$. This shows that
\begin{equation}\label{e3.7}
\mathfrak B^*_{\max(\kappa_1,\kappa_2)}(Q)\subseteq \mathfrak B^*_{\kappa_1}(Q)\cup \mathfrak B^*_{\kappa_2}(Q).
\end{equation}

Equations \eqref{e3.6}, \eqref{e3.7} imply the required statement.
\end{proof}

\subsection{Proof of Corollary \ref{cor2.10}}
\begin{proof}
Map $\Sigma_*:(\mathscr{B}^*(Q),\cup,\subseteq)\rightarrow  (\mathscr B(Q),\cup,\subseteq)$ is given by the formula
\[
\Sigma_*(U):=\mathfrak B_{\Sigma(\kappa)}(Q),\qquad U=\mathfrak B^*_\kappa(Q)\in \mathscr{B}^*(Q).
\]
It is well defined due to Theorem \ref{prop2.2}\,(1). It is bijective due to Theorem \ref{prop2.2}\,(2),(3). Since map $\Sigma$ is additive (see \eqref{equ2.1}), $\Sigma_*$ is the homomorphism of abelian semigroups.

Finally, $\Sigma_*$ preserves partial orders because of the maximality property of $\mathfrak{B}_{\Sigma(\kappa)}(Q)$ in $\mathfrak{B}^*_{\kappa}(Q)$ proved in Theorem \ref{prop2.2}\,(1). Indeed, if 
$\mathfrak{B}_{\kappa_1}^*(Q)\subseteq \mathfrak{B}_{\kappa_2}^*(Q)$, then $\mathfrak{B}_{\Sigma(\kappa_1)}(Q)\subseteq  \mathfrak{B}_{\kappa_2}^*(Q)$ and therefore by that property $\mathfrak{B}_{\Sigma(\kappa_1)}(Q)\subseteq \mathfrak{B}_{\Sigma(\kappa_2)}(Q)$. If, in addition, $\mathfrak{B}_{\kappa_1}^*(Q)\neq \mathfrak{B}_{\kappa_2}^*(Q)$, then $\mathfrak{B}_{\Sigma(\kappa_1)}(Q)\neq  \mathfrak{B}_{\Sigma(\kappa_2)}(Q)$ by Theorem \ref{prop2.2}\,(2), as required.
\end{proof}

\sect{Proofs of Theorems \ref{te1}, \ref{mazur2} and \ref{mar2}}
\subsection{Proof of Theorem \ref{te1}}
\begin{proof}
We prove part (2) only; the proof of part (1) is similar. 

The proof is a simple corollary of Theorem \ref{prop2.1}\,(1).  Indeed, since  $(A\hat{\otimes}_\alpha B)^{-1}$ is an open subset of $A\hat\otimes_\alpha B$ and $\mathfrak B_{\varphi}(Q)$, $Q=((A,V);B;\alpha)$, is comeager in  $A\hat\otimes_\alpha B$,  set
$S:=\mathfrak B_{\varphi}(Q)\cap (A\hat{\otimes}_\alpha B)^{-1}$  is comeager in $(A\hat{\otimes}_\alpha B)^{-1}$.  Next, since map $^{-1}: (A\hat{\otimes}_\alpha B)^{-1}\rightarrow (A\hat{\otimes}_\alpha B)^{-1}$, $g\mapsto g^{-1}$, is a homeomorphism, set $S':=\{g\in (A\hat{\otimes}_\alpha B)^{-1}\, :\, g^{-1}\in S\}$ is comeager in $(A\hat{\otimes}_\alpha B)^{-1}$ . Hence, $S\cap S'$ is comeager in $(A\hat{\otimes}_\alpha B)^{-1}$ as well. Similarly, for each $g\in (A\hat{\otimes}_\alpha B)^{-1}$ set $g\cdot (S\cap S')$ is comeager in $(A\hat{\otimes}_\alpha B)^{-1}$. Therefore,
$\bigl(g\cdot (S\cap S')\bigr)\cap (S\cap S')$ is comeager in $(A\hat{\otimes}_\alpha B)^{-1}$ also. This implies the required statement.
\end{proof}
\subsection{Proof of Theorem \ref{mazur2}}
\begin{proof}
First, we will prove the theorem for $\mathbb F=\mathbb C$. In this case, the hypotheses imply that $U_\mathbb C\subset\mathbb C^N$ is a Stein domain and $X_\mathbb C$ is its Stein submanifold.  Therefore there exists a holomorphic retraction $r:N\rightarrow X_\mathbb C$ of an open neighbourhood $N\subset U_\mathbb C$ of $X_\mathbb C$ (see, e.g., \cite{GR} for basic results of the theory of Stein spaces).
Also, $X_\mathbb C$ is defined as the set of common zeros of a family of holomorphic functions on $U_\mathbb C$. Then, since $U_\mathbb C$ admits an exhaustion by compact polynomially convex subsets, the Runge approximation theorem implies that $X_\mathbb C$
admits an exhaustion by compact polynomially convex subsets as well. In particular, for each bounded subset of $X_\mathbb C$ its polynomially convex hull belongs to $X_\mathbb C$.  

Let $O\Subset X_\mathbb C$ be an open  bounded subset and $\widehat O$ its polynomially convex hull. Then there exists a {\em Weil polynomial polyhedron} 
\[
W=\bigl\{z\in\mathbb C^N\, :\, \max_{1\le i\le k}|p_i(z)|<1\ {\rm all}\ p_i\in\mathcal P(\mathbb C^N)\bigr\}
\]
such that ${\rm cl}(W)\Subset N$ and $\widehat O\subset W$. 

Further, since $W$ is open, set $W(A)$ of maps $g: M_A\rightarrow \mathbb C^N$ whose coordinate functions are in $A\, (:=A_\mathbb C)$ and images are in $W$ is an open subset of $A\otimes\mathbb C^N$ (we use here that the Gelfand transform is a nonincreasing norm morphism of algebras). In particular, due to the argument of the proof of Theorem \ref{te1} (cf. \cite{Ma}), sets $W(A)\cap\mathfrak{B}_{\varphi}((A,V);\mathbb C^N)$ are comeager in $W(A)$ for all scale functions $\varphi$. 

Next, by $V_i'\subset C(M_A;r(W))$ we denote set $\{r\circ g\, :\,  g\in (V_i\otimes\mathbb C^N)\cap W(A)\}$. Using functional calculus for commutative Banach algebras (see, e.g., \cite[Ch.\,3.4]{G}), one obtains that $V_i'\subset X_\mathbb C(A)$.
Then maps from $V':=\cup_i V_i'$ are dense in $O(A)\, (:= X_\mathbb C(A)\cap C(M_A;O))$. 
\begin{Lm}\label{lem1}
There exist numbers $t\in (0,1)$ and $n_*\in\mathbb N$ such that $E_{n\cdot j}(h)\le t^n$ for all $h\in V_j'$, $j\in\mathbb N$, and $n\ge n_*$.
\end{Lm}
\begin{proof}
Let $d:=\max_{1\le i\le k}{\rm deg}(p_i)$.
As follows from the Weil integral  representation formula \cite{W} for the coordinate functions of the map 
$r|_{W}=(r_1,\dots, r_N):W\rightarrow\mathbb C^N$,  
\begin{equation}\label{weil}
r_i(z)=\sum_{\alpha}p_{i\alpha}(z)\cdot p_1^{\alpha_1}(z)\cdots p_k^{\alpha_k}(z),\quad z\in W,\quad \alpha=(\alpha_1,\dots,\alpha_k)\in\mathbb Z_+^k,
\end{equation}
where all $p_{i\alpha}\in\mathcal P(\mathbb C^N)$,   $ {\rm deg}(p_{i\alpha})\le N\cdot (d-1)$, $\sup_{i,\alpha}\bigl(\sup_W |p_{i\alpha}|\bigr)\le C$,
and the series converges uniformly on compact subsets of $W$. \smallskip

If $g=(g_1,\dots, g_N)\in (V_j\otimes\mathbb C^N)\cap W(A)$, then (recall that $V\subset A$ is a filtered subalgebra)  $p_l\circ g\in V_{d\cdot j}$, $1\le l\le k$, so that $(p_1\circ g)^{\alpha_1}\cdots (p_k\circ g)^{\alpha_k}\in V_{|\alpha|\cdot d\cdot j}$ (here $|\alpha|:=\alpha_1+\cdots +\alpha_k)$. Also, $p_{i\alpha}\circ g\in V_{d\cdot (N-1)\cdot j}$ for all $i,\alpha$.  Note that $g(M_A)$ is a compact subset of $W$; hence,
\begin{equation}\label{maximum}
\max_{1\le l\le k}\sup_{g(M_A)}|p_l |<1.
\end{equation}
Since the {\em spectral radius} of $p_l\circ g\in A$ is
\[
\lim_{n\rightarrow\infty}\sqrt[n]{\|(p_l\circ g)^n\|_A}=\sup_{g(M_A)}|p_l |,
\]
\eqref{maximum} implies that there exist constants $c>0$ and $\rho\in (0,1)$ such that for all $n\in\mathbb N$, $1\le l\le k$,
\begin{equation}\label{crho}
\|(p_l\circ g)^n\|_A\le c\rho^n.
\end{equation}
Also, since all norms on the finite-dimensional Banach space $V_{d\cdot (N-1)\cdot j}$ are equivalent, there exists some $C'>0$ such that
\begin{equation}\label{equiv}
\sup_{i,\alpha}\|p_{i\alpha}\circ g\|_A\leq C'\sup_{i,\alpha}\left(\sup_{M_A}|p_{i,\alpha}\circ g|\right)\le C'\sup_{i,\alpha}\left(\sup_W |p_{i\alpha}|\right)\le C'\cdot C.
\end{equation}
Inequalities \eqref{crho} and  \eqref{equiv} imply that series
\[
\sum_{\alpha}(p_{i\alpha}\circ g)\cdot (p_1\circ g)^{\alpha_1}\cdots (p_k\circ g)^{\alpha_k}
\]
converges absolutely in $A$ to element $r_i\circ g=:h_i$, $1\le i\le N$. Moreover, for
$h=(h_1,\dots, h_N):=r\circ g\in V_j'$
there exist 
numbers $t\in (0,1)$ and $n_*\in\mathbb N$ such that for all $n\ge n_*$,
\[
\begin{array}{l}
\displaystyle
E_{n\cdot j}(h):=\sqrt{ \sum_{i=1}^N \bigl(E_{n\cdot j}(h_i)\bigr)^2}\le\sum_{i=1}^N\bigl\| h_i-\sum_{|\alpha|=0}^{\lfloor\frac nd\rfloor- (N-1)} (p_{i\alpha}\circ g) \cdot  (p_1\circ g)^{\alpha_1}\cdots (p_k\circ g)^{\alpha_k}  \bigr\|_A \\
\\
\displaystyle \le \sum_{i=1}^N\sum_{|\alpha|=\lfloor\frac nd\rfloor- (N-2)}^\infty  C'\cdot C\cdot c^k\cdot\rho^{|\alpha|}= \sum_{i=1}^N\sum_{l=\lfloor\frac nd\rfloor- (N-2)}^\infty C'\cdot C\cdot c^k \cdot {l \choose k-1 }\cdot \rho^{l}<t^n,
\end{array}
\]
as required.
\end{proof}

In the subsequent proof we use the following result.

\begin{Lm}\label{lem3.2}
Given a nonincreasing function $\varphi:\mathbb N\rightarrow (0,\infty)$ such that $\lim_{n\rightarrow\infty}\varphi(n)=0$ and $\lim_{n\rightarrow\infty}n\cdot\varphi(n)=\infty$ there exists a nondecreasing function $\xi:\mathbb N\rightarrow \mathbb N$, $\lim_{n\rightarrow\infty}\xi(n)=\infty$, such that for all sufficiently large $n$
\[
\varphi(1)\le \xi(n)\cdot\varphi(n\cdot\xi(n))\le 2\varphi(1).
\]
\end{Lm}
\begin{proof}
Without loss of generality we may assume that $\varphi(1)=1$.
Set $\theta(n):=n\cdot\varphi(n)$, $n\in\mathbb N$. By the definition, function $y\mapsto \frac{\theta(n\cdot y)}{n}$, $y\in\mathbb N$, tends to $\infty$ as $y\rightarrow\infty$ and $\frac{\theta(n\cdot 1)}{n}\le 1$. We define
\[
\xi(n):=\min\left\{m\in\mathbb N\, :\, \frac{\theta(n\cdot m)}{n}\ge 1\right\}.
\]
Then $\lim_{n\rightarrow\infty} \xi(n)=\infty$.  (For otherwise, $1\le \lim_{n\rightarrow\infty}\frac{\theta(n\cdot \xi(n))}{n}=0$, a contradiction.) Moreover, $\xi$ is a nondecreasing function. Indeed, since $\varphi$ is nonincreasing,
\[
1\le \frac{\theta\bigl((n+1)\cdot \xi(n+1)\bigr)}{n+1}\le \frac{\theta(n\cdot \xi(n+1))}{n}.
\]
Then $\xi(n)\le \xi(n+1)$ by the definition of $\xi(n)$. Finally, if $\xi(n)>1$, then
\[
\frac{1}{\xi(n)}\le \frac{\theta(n\cdot \xi(n))}{n\cdot \xi(n)}\le \frac{\theta(n\cdot (\xi(n)-1))}{n\cdot (\xi(n)-1)}<\frac{1}{\xi(n)-1}.
\]
This implies
\[
1\le \xi(n)\cdot\varphi(n\cdot\xi(n))\le \frac{\xi(n)}{\xi(n)-1}\le 2,
\]
as required.
\end{proof}

Further, let $\psi:\mathbb N\rightarrow (0,\infty)$ be a nonincreasing function such that $\lim_{n\rightarrow\infty}\psi(n)=0$.
By $\mathfrak{B}_{\psi}((A, V');X_\mathbb C)$ we denote the set of  
$g\in X_\mathbb C(A)$ such that  
\[
\varliminf_{n\rightarrow\infty}(E_n'(g))^{\psi(n)}<1,
\]
where $E_n'(g):=\inf_{h\in V_n'}\|g-h\|_{A\otimes\mathbb C^N}$ is the distance from $g\in X_\mathbb C(A)$ to $V_n'$.

\noindent As in the proof of Theorem \ref{te1} (cf. \cite{Ma}) one shows that $\mathfrak{B}_{\psi}((A, V');X_\mathbb C)$ is comeager in $X_\mathbb C(A)$ (recall that the latter is a closed subset of  Banach space $A\otimes\mathbb C^N$, and therefore is complete in the induced metric).

By definition, for $g\in \mathfrak{B}_{\psi}((A,V');X_\mathbb C)\cap C(M_A;O)$ there exist $\{n_l\}\subset \mathbb N$ and $g_{l}\in V_{n_l}'\cap C(M_A;O)$ such that $\|g-g_l\|_{A\otimes\mathbb C^N}\le c^{\frac{1}{\psi(n_l)}}$ for some $c\in (0,1)$. Then Lemma \ref{lem1} yields
\begin{equation}\label{compos}
E_{n\cdot n_l}(g)\le \|g-g_l\|_{A\otimes\mathbb C^N}+ E_{n\cdot n_l}(g_l)\le c^{\frac{1}{\psi(n_l)}}+t ^n\quad {\rm for\ all}\quad n\ge n_*.
\end{equation}

Taking here $\psi:=\frac{1}{\xi}$ with $\xi$ as in Lemma \ref{lem3.2} and
$n:=\xi(n_l)$ (for all sufficiently large $l$ such that 
$\xi(n_l)\ge n_*$) we obtain from this lemma for some $\tilde c\in (0,1)$,
\[
E_{n_l\cdot\xi(n_l)}(g)\le 2\cdot\bigl(\max(c,t)\bigr)^{\,\xi(n_l)}\le  2\cdot\bigl(\max(c,t)\bigr)^{\frac{\varphi(1)}{\varphi(n_l\cdot\xi(n_l))}}\le \tilde c^{\frac{1}{\varphi(n_l\cdot\xi(n_l))}}.
\]
The latter implies that
\[
\varliminf_{k\rightarrow\infty}\bigl(E_k(g)\bigr)^{\varphi(k)}\le \tilde c,
\]
that is, $g\in \mathfrak{B}_{\varphi}((A,V);\mathbb C^N)$. 

Thus, we have proved that $\mathfrak{B}_{\psi}((A, V');X_{\mathbb C})\cap C(M_A;O)\subset\mathfrak{B}_{\varphi} ((A,V);\mathbb C^N)\cap X_\mathbb C(A)$ for an arbitrary open $O\Subset X_{\mathbb C}$.

Now, choosing an exhaustion $O_1\subsetneq O_2\subsetneq\cdots$ of $X_{\mathbb C}$ by open bounded sets and taking into account that
\[
\mathfrak{B}_{\psi}((A, V');X_{\mathbb C})=\bigcup_{i=1}^\infty \mathfrak{B}_{\psi}((A, V');X_{\mathbb C})\cap C(M_A;O_i),
\]
we obtain that $\mathfrak{B}_{\psi}((A, V');X_{\mathbb C})\subset \mathfrak{B}_{\varphi} ((A,V);\mathbb C^N)\cap X_\mathbb C(A)$. 

Hence, set
$\mathfrak{B}_{\varphi} ((A,V);\mathbb C^N)\cap X_\mathbb C(A)$ is comeager in $X_{\mathbb C}(A)$. This completes the proof of part (a) of the theorem for $\mathbb F=\mathbb C$.\smallskip

To prove part (b),  we choose $\psi(n):=\frac{1}{2^n}$. Then \eqref{compos} with $n:=2^{n_l}$ implies, for some $\tilde c\in (0,1)$ and all sufficiently large $l$,
\begin{equation}\label{eq3.6}
E_{n_l\cdot 2^{n_l}}(g)\le \tilde c^{2^{n_l}}.
\end{equation}
Observe that for function 
$\tilde\varphi(n):=\frac{1+\ln n}{n}$, $n\in\mathbb N$,
\[
\frac{\ln 2}{2^{n_l}}\le \tilde\varphi(n_l\cdot 2^{n_l})\le \frac{1+\ln 2}{2^{n_l}}.
\]
From here and \eqref{eq3.6} we obtain, for $\bar c:=\tilde c^{\,\ln 2}$ and all sufficiently large $l$,
\[
E_{n_l\cdot 2^{n_l}}(g)\le \bar c^{\frac{1}{\tilde\varphi(n_l\cdot 2^{n_l})}}.
\]
This implies that $g\in \mathfrak{B}_{\tilde\varphi}((A,V);\mathbb C^N)$. In addition, let us prove that $g\in \mathfrak{B}^*((A,V);\mathbb C^N)$. Indeed, 
(\ref{compos})  leads to the inequality
\[
E_{n\cdot n_l}(g)\le 2\cdot C^{\min(2^{n_l},\,n)}, \quad {\rm where}\quad C:=\max(c,t).
\]
Thus, since $\{E_k(g)\}_{k\in\mathbb N}$ is a nonincreasing sequence and $\ln E_{n\cdot n_l}(g)<0$ for all sufficiently large $l$, for such $l$ and all  $1\le n\le 2^{n_l}$ we obtain
\[
\begin{array}{l}
\displaystyle
\sum_{i=0}^{n_l-1}\frac{\ln E_{n\cdot n_l+i}(g)}{(n\cdot n_l+i)^2}\le \ln E_{n\cdot n_l}(g)\cdot\sum_{i=0}^{n_l-1}\frac{1}{(n\cdot n_l+i)^2}\le \ln E_{n\cdot n_l}(g)\cdot\left(\frac{1}{n\cdot n_l}-\frac{1}{(n+1)\cdot n_l}\right)\\
\\
\displaystyle = \frac{\ln E_{n\cdot n_l}(g)}{n\cdot (n+1)\cdot n_l}\le \frac{\ln 2}{n\cdot (n+1)\cdot n_l}+\frac{\ln C}{(n+1)\cdot n_l}.
\end{array}
\]
Summing these inequalities over $n$ we get
\[
\begin{array}{l}
\displaystyle
\sum_{j=n_l}^{2^{n_l}\cdot n_l-1}\frac{\ln E_j(g)}{j^2}=\sum_{n=1}^{2^{n_l}-1}\sum_{i=0}^{n_l-1}\frac{\ln E_{n\cdot n_l+i}(g)}{(n\cdot n_l+i)^2}\le \sum_{n=1}^{2^{n_l}-1}\frac{\ln 2}{n\cdot (n+1)\cdot n_l}+\sum_{n=1}^{2^{n_l}-1}\frac{\ln C}{(n+1)\cdot n_l}\\
\\
\displaystyle
\le \frac{\ln 2}{n_l}+\ln 2\cdot\ln C.
\end{array}
\]
Hence,  there exists $L\in\mathbb N$ such that for all $l\ge L$ the right-hand side of the previous inequality does not exceed $\frac{\ln 2\cdot\ln C}{2}$. Passing to a subsequence, if necessary, we may assume that $2^{n_l}\cdot n_l<n_{l+1}$ for all $l$. Then we have
\[
\sum_{j=1}^\infty\frac{\ln E_j(g)}{j^2}\le \sum_{l=L}^\infty\sum_{j=n_l}^{2^{n_l}\cdot n_l-1}\frac{\ln E_j(g)}{j^2}\le \sum_{l=L}^\infty \frac{\ln 2\cdot\ln C}{2}=-\infty.
\] 
This shows that $g\in \mathfrak{B}^*((A,V);\mathbb C^N)$, as required.

Thus we have proved that  
\[
\mathfrak{B}_{\psi}((A, V');X_{\mathbb C})\subset \bigl(\mathfrak{B}_{\tilde\varphi} ((A,V);\mathbb C^N)\cap \mathfrak{B}^*((A,V);\mathbb C^N)\bigr)\cap X_\mathbb C(A).
\]
In particular, the latter set is comeager in $X_\mathbb C(A)$ which completes the proof of part (b) for $\mathbb F=\mathbb C$.\smallskip

Now, let us discuss the case of $\mathbb F=\mathbb R$. Here $X_\mathbb R$ is a closed analytic submanifold of the domain  $U_\mathbb R\subset\mathbb R^N$.
Then there exists an open neighbourhood $N\subset U_\mathbb R$ of $X_\mathbb R$ and an analytic retraction $r: N\rightarrow X_\mathbb R$. Due to analyticity, $r$ admits an extension $r': N'\rightarrow\mathbb C^N$, where $N'\subset\mathbb C^N$ is an open neighbourhood of $X_\mathbb R$ in $\mathbb C^N$ containing $N$ and $r'$ is holomorphic on $N'$. Further, since $X_\mathbb R\subset\mathbb R^N$, each compact subset of $X_\mathbb R$ is polynomially convex in $\mathbb C^N$. Hence, for an open bounded subset $O$ of $X_\mathbb R$ its polynomially convex hull belongs to $X_\mathbb R$ and  therefore there exists a Weil polynomial polyhedron $W$ such that ${\rm cl}(W)\Subset N'$ and $\widehat O\subset W$.
The rest of the proof repeats literally the corresponding proof for the case $\mathbb F=\mathbb C$ (the only difference is that now, in the proof of the analog of Lemma \ref{lem1}, one applies the Weil integral representation formula to coordinate functions of map $r'|_W$).
\end{proof}

\subsection{Proof of Theorem \ref{mar2}}
\begin{proof}
It is sufficient to show that $X_\mathbb F(A_\mathbb F)$ is closed under group operations on $C(M_A;X_\mathbb F)$. Then the rest of the proof is based on Theorem \ref{mazur2} and repeats arguments of the proof of Theorem \ref{te1}. 

We prove the required result for $\mathbb F=\mathbb C$. For $\mathbb F=\mathbb R$ the proof is similar (cf. the end of the proof of Theorem \ref{mazur2}). 

We retain notation of the proof of Theorem \ref{mazur2}. So, let $r:N\rightarrow X_\mathbb C$ be the  holomorphic retraction of the open neighbourhood $N\subset U_\mathbb C$.  By definition, multiplication on $X_\mathbb C$ is given by a holomorphic map $S: X_\mathbb C\times X_\mathbb C \rightarrow X_\mathbb C$.  Consider holomorphic map $\tilde S: N\times N\rightarrow \mathbb C^N$, $\tilde S(z,w):=S(r(z), r(w))$. If $g,h\in X_\mathbb C(A)$, then 
$(g,h):M_A\rightarrow  N\times N\subset \mathbb C^{2N}$ is a continuous map whose coordinate functions belong to $A$. Since for each compact subset of $X_\mathbb C\times X_\mathbb C$ its polynomially convex hull in $\mathbb C^{2N}$ belongs to $X_\mathbb C\times X_\mathbb C$, applying functional calculus for $A$ (see, e.g., \cite[Ch.\,3.4]{G}) to $\tilde S$ and $g,h$ we obtain
that $\tilde S(g,h)=S(g,w)\in X_\mathbb C(A)$ for all $g,h\in X_\mathbb C(A)$. Thus, $X_\mathbb C(A)$ is closed under multiplication on  $C(M_A;X_\mathbb C)$. A similar argument shows that $X_\mathbb C(A)$ is closed under the operation of taking the inverse. 
\end{proof}

\sect{Proofs of Results of Subsection~2.3}
\subsection{Proof of Theorem \ref{te1.9}}
\begin{proof}
By definition and since $\|\cdot\|_A\ge\|\cdot\|_{C(M)}$ and the injective crossnorm $\epsilon$ does not exceed the uniform crossnorm $\alpha$ on $A\hat\otimes_\alpha B$, for an element $f\in\mathfrak{B}_\varphi((A,V);B;\alpha)$ there exist a number $\rho\in (0,1)$ and a sequence $\{f_{n_i}\}_{i\in\mathbb N}\subset V$, $f_{n_i}\in V_i\otimes B$, such that 
\[
\gamma_{n_i}:=\sup_{s\in S}\|f(s)-f_{n_i}(s)\|_B\le  \|f-f_{n_i}\|_{A\hat\otimes_\alpha B}\le \rho^{\frac{1}{\varphi(n_i)}}.
\]
Without loss of generality we may assume that
\begin{equation}\label{eq3.1}
\max\left(\sup_{m\in M}\|f_{n_i}(m)\|_B, \sup_{m\in M}\|f(m)\|_B\right)=1.
\end{equation}
Next, according to \eqref{eq1.5}, 
\[
c_i:=Cov(S,\gamma_{n_i})\le C\cdot\left(\frac{1}{\gamma_{n_i}}\right)^k.
\]
Let $\bigl(B_{ij}\bigr)_{1\le j\le c_i}$, $B_{ij}:=B_{\gamma_{n_i}}(m_{ij})\subset M$, be an open cover of $S$ corresponding to $c_i$. Then, due to \eqref{eq3.1} and by the definition of Markov constants, for each $j$ and all $m_1,m_2\in B_{ij}$,
\begin{equation}\label{eq3.2}
\begin{array}{l}
\displaystyle
\|f_{n_i}(m_1)-f_{n_i}(m_2)\|_B=\sup_{b^*\in B^*,\, \|b^*\|_{B^*}\le 1}|b^*(f_{n_i}(m_1))-b^*(f_{n_i}(m_2))|\\
\\
\displaystyle
\le \mathcal M_{V_{n_i}}(S)\cdot \sup_{m\in M}\|f_{n_i}(m)\|_B\cdot d(m_1,m_2)\le 2\mathcal M_{V_{n_i}}(S)\cdot\gamma_{n_i} .
\end{array}
\end{equation}
In particular, since 
\[
{\rm dist}\bigl(\Gamma_f(S),\Gamma_{f_{n_i}}(S)\bigr):=\inf\left\{d_{M\times B}(m_1,m_2)\, :\, m_1\in \Gamma_f(S),\, m_2\in \Gamma_{f_{n_i}}(S)\right\}
\le \gamma_{n_i},
\]
inequality \eqref{eq3.2} implies that $\Gamma_f(S)$ can be covered by sets $\{W_{ij}\}_{1\le j\le c_i}$, $W_{ij}:=\bigl\{(m,b)\in B_{ij}\times B\, :\, \|b-f_{n_i}(m)\|_B\le \gamma_{n_i} \bigr\}$,
of diameter at most $2\mathcal M_{V_{n_i}}(S)\cdot\gamma_{n_i}$.
Thus, for $\psi_k(t):=t^k\cdot\psi(t)$, $t\in [0,\infty)$, we have
\begin{equation}\label{long}
\begin{array}{lr}
\displaystyle
\sum_{j=1}^{c_i} \psi_k\bigl({\rm diam}(W_{ij})\bigr)\\
\medskip
\displaystyle \le c_i\cdot\psi_k(2\mathcal M_{V_{n_i}}(S)\cdot\gamma_{n_i})\le C\cdot\left(\frac{1}{\gamma_{n_i}}\right)^k\cdot (2\mathcal M_{V_{n_i}}(S)\cdot\gamma_{n_i})^k\cdot\psi(2\mathcal M_{V_{n_i}}(S)\cdot\gamma_{n_i})
\\
\medskip
\displaystyle 
\le C\cdot (2\mathcal M_{V_{n_i}}(S))^k\cdot \psi\left(2\mathcal M_{V_n}(S)\cdot\rho^{\frac{1}{\varphi(n)}}\right).
\end{array}
\end{equation}
From here and hypothesis \eqref{eq1.6} we obtain
\begin{equation}\label{eq3.3}
\varliminf_{i\rightarrow\infty}\sum_{j=1}^{c_i} \psi_k\bigl({\rm diam}(W_{ij})\bigr)\le C\cdot 2^k\cdot L(\rho')\quad {\rm for\ all}\quad \rho'\in (\rho,1).
\end{equation} 
Recall that 
\[
\begin{array}{l}
\displaystyle
\mathcal H^{\psi_k}_{M\times B}\bigl(\Gamma_f(S)\bigr)=\lim_{\delta\rightarrow 0^+} \mathcal H_{\delta\, M\times B}^{\psi_k}\bigl(\Gamma_f(S)\bigr),\quad {\rm where}\\
\\
\displaystyle  \mathcal H_{\delta\, M\times B}^{\psi_k}\bigl(\Gamma_f(S)\bigr)=\inf\left\{\sum_{j=1}^\infty\psi_k({\rm diam}(U_i))\, :\, \Gamma_f(S)\subset\bigcup_{i=1}^\infty U_i,\, \forall i\ {\rm diam}(U_i)\le\delta\right\}.
\end{array}
\]
Since, due to \eqref{eq1.6}, 
\[
\lim_{i\rightarrow\infty}\max_{1\le j\le c_i}\bigl\{{\rm diam}(W_{ij})\bigr\}\le\lim_{i\rightarrow\infty}2\mathcal M_{V_{n_i}}(S)\cdot\gamma_{n_i}\le \lim_{i\rightarrow\infty}2\mathcal M_{V_{n_i}}(S)\cdot\rho^{\frac{1}{\varphi(n_i)}}=0,
\]
inequality \eqref{eq3.3} implies that
\[
\mathcal H^{\psi_k}_{M\times B}\bigl(\Gamma_f(S)\bigr)\le C\cdot 2^k\cdot L(\rho')\quad {\rm for\ all}\quad \rho'\in (\rho,1).
\]
This completes the proof of the theorem.
\end{proof}

\subsection{Proof of Corollary \ref{cor1.10}}
\begin{proof}
(1) Since $f(S)$ is the image of $\Gamma_f(S)$ under the Lipschitz map $\pi_B: M\times B\rightarrow B$, $\pi_B(m,b)= b$, with the {\em Lipschitz constant} one, Theorem \ref{te1.9} implies
\[
\mathcal H^{\psi_k}_{B}\bigl(f(S)\bigr)=\mathcal H^{\psi_k}_{B}\bigl(\pi_B(\Gamma_f(S))\bigr)\le \mathcal H^{\psi_k}_{M\times B}\bigl(\Gamma_f(S)\bigr)<\infty.
\]
(2) The required statement follows straightforwardly from the next general result.

{\em Let $X$ be a compact metric space and let $d\in\mathbb N$. Let $\psi, \sigma$ be continuous gauge functions such that $\sigma(t)=\psi(t)\cdot t^d$, $t\in (0,\infty)$. If $g:X\rightarrow\mathbb R^d$ is Lipschitz with the Lipschitz constant $Lip(g)$, then
\begin{equation}\label{eq3.4}
\int_{\mathbb R^d}\mathcal H_X^{\psi}\bigl(g^{-1}(y)\bigr)d\lambda^d(y)\le c(d)\cdot\bigl(Lip(g)\bigr)^d\cdot\mathcal H_X^\sigma (X);
\end{equation}
here $c(d)$ is a constant depending on $d$ only and $\lambda^d$ is Lebesgue measure on $\mathbb R^d$.}

(For the proof, see, e.g., \cite[Th.\,7.7]{Mat} or 
\cite[Prop.\,7.9]{F} with natural modifications.) 

Now, to obtain our result we choose $X=\Gamma_f(S)$, $\psi=\psi_{k,d}$, $\sigma=\psi_{k}$ and $g=\pi_B$.  Then \eqref{eq3.4} and Theorem \ref{te1.9} imply that
\[
\int_{\mathbb R^d}\mathcal H_{M\times B}^{\psi_{k,d}}\bigl(g^{-1}(y)\bigr)d\lambda^d(y)<\infty,
\]
i.e., $\mathcal H_{M\times B}^{\psi_{k,d}}\bigl(g^{-1}(y)\bigr)<\infty$ a.e. $y\in\mathbb R^d$. Finally, $f^{-1}(y)\cap S$ is the image of $g^{-1}(y)$ under the Lipschitz map $\pi_S:\Gamma_f(S)\rightarrow S$, $\pi_S(m,b)=m$, with the Lipschitz constant $\le 1$. Thus,
\[
\mathcal H_{M}^{\psi_{k,d}}\bigl(f^{-1}(y)\cap S\bigr)=\mathcal H_{M}^{\psi_{k,d}}\bigl(\pi_M(g^{-1}(y))\bigr)\le \mathcal H_{M\times B}^{\psi_{k,d}}\bigl(g^{-1}(y)\bigr)<\infty\quad  a.e.\quad y\in\mathbb R^d,
\]
as required.
\end{proof}
\subsection{Proof of Theorem \ref{te1.13}}
\begin{proof}
The proof is based on the following result.
\begin{Lm}\label{lem3.1}
Suppose $f\in \mathfrak{B}^*((A,V);B;\alpha)$. Then for $\psi$ satisfying \eqref{eq1.9} and each sequence $\{b_n\}_{n\in\mathbb N}\subset (0,\infty)$ such that series $\sum_{n=1}^\infty \frac{\ln b_n}{n^2}$ converges,  
\[
\varliminf_{n\rightarrow\infty} n^s\cdot \psi\bigl(b_n\cdot E_n(f)\bigr)=0.
\]
\end{Lm}
\begin{proof}
Assuming the contrary we obtain for some $c>0$,
\[
n^s\cdot \psi\bigl(b_n\cdot E_n(f)\bigr)\ge c.
\]
Since $\psi$ is continuous and increasing, the previous inequality implies
\[
\frac{\ln\bigl(E_n(f)\bigr)}{n^2}\ge \frac{\ln\left(\psi^{-1}\left(\frac{c}{n^s}\right)\right)-\ln b_n}{n^2}.
\]
Without loss of generality we may assume that $c$ is so small that $\psi^{-1}\left(\frac{c}{n^s}\right)<1$ for all $n\ge 1$; then since $\psi^{-1}$ increases,
\[
\begin{array}{l}
\displaystyle
\sum_{n=2}^\infty\frac{\ln\left(\psi^{-1}\left(\frac{c}{n^s}\right)\right)}{n^2}\ge \sum_{n=2}^\infty\ln\left(\psi^{-1}\left(\frac{c}{n^s}\right)\right)\cdot\left(\frac{1}{n-1}-\frac{1}{n}\right)\ge\int_0^1\ln\bigl(\psi^{-1}(ct^s)\bigr)\, dt\\
\\
\displaystyle =\frac{1}{c}\cdot \int_0^{\psi^{-1}(c^s)} \ln y\cdot \bigl(\psi^{\frac 1s}\bigr)'(y) dy=\ln\bigl(\psi^{-1}(c^s)\bigr)-\frac{1}{c}\cdot\int_0^{\psi^{-1}(c^s)}\frac{\bigl(\psi(y)\bigr)^{\frac 1s}}{y}\, dy>-\infty .
\end{array}
\]
In this formula we used the substitution $t=\frac{(\psi(y))^{\frac 1s}}{c}$ and then integration by parts where $\lim_{y\rightarrow 0^+}\ln y\cdot (\psi(y))^{\frac 1s}=0$ due to condition \eqref{eq1.9}.

Therefore we obtain
\[
\sum_{n=1}^\infty\frac{\ln\bigl(E_n(f)\bigr)}{n^2}>-\infty.
\]
This contradicts the definition of class $\mathfrak B^*((A,V);B;\alpha)$, see \eqref{beurling}.
\end{proof}

Now, for $f\in \mathfrak{B}^*((A,V);B;\alpha)$ we choose a sequence $\{f_n\}_{i\in\mathbb N}$, $f_n\in V_n\otimes B$, such that $\gamma_n:=\sup_{s\in S}\|f(s)-f_n(s)\|_B\le 2 E_n(f)$ (cf. the proof of Theorem \ref{te1.9}). Then we apply Lemma \ref{lem3.1} with $b_n=4 n^s$, $n\in\mathbb N$, to find a subsequence $\{n_i\}_{i\in\mathbb N}\subset\mathbb N$ such that 
\begin{equation}\label{eq3.7}
\lim_{i\rightarrow\infty} n_i^s\cdot \psi\bigl(4 n_i^s\cdot E_{n_i}(f)\bigr)=0.
\end{equation}
The rest of the proof repeats word-for-word the proofs of Theorem \ref{te1.9} and Corollary \ref{cor1.10} starting from formula \eqref{eq3.1}, where instead of \eqref{long} we have, in notation of these results,
\[
\begin{array}{l}
\sum_{j=1}^{c_i} \tilde\psi_k\bigl({\rm diam}(W_{ij})\bigr)\\
\medskip
\displaystyle \le c_i\cdot\tilde \psi_k(2\mathcal M_{V_{n_i}}(S)\cdot\gamma_{n_i})\le C\cdot\left(\frac{1}{\gamma_{n_i}}\right)^k\cdot \bigl((2 n_i^s\cdot\gamma_{n_i})\cdot\psi\bigl(4n_i^s\cdot E_{n_i}(f)\bigl)\bigr)^{k}\\
\medskip
\displaystyle 
=C\cdot \bigl(2 n_i^s\cdot\psi\bigl(4n_i^s\cdot E_{n_i}(f)\bigl)\bigr)^{k}\rightarrow 0\quad {\rm as}\quad i\rightarrow\infty.
\end{array}
\]
We leave details to the reader.
\end{proof}

\subsection{Proof of Corollary \ref{cor2.16}}
\begin{proof}
Let $B\Subset\mathbb C^{N}$ be an open Euclidean ball containing $K$. For a compact set $C\subset B$ by ${\rm cap}(C;B)$ we denote the capacity of $S$ relative to $B$ determined in terms of the Monge-Ampere operator (see \cite{BT}). Since $f\in C(K)$, the set function ${\rm range}(f)\ni c\mapsto f^{-1}(c)$ is upper-semicontinuous with respect to convergence in the Hausdorff metric on compact subsets of $\mathbb C^{N}$ meaning that if a sequence $\{c_n\}\subset {\rm range}(f)$ converges to $c$, then any convergent subsequence of the sequence $\{f^{-1}(c_n)\}$ converges to a compact subset of $f^{-1}(c)$. This and the Choquet capacitary property for a decreasing sequence of compact sets imply that function $C_f(z):={\rm cap}(f^{-1}(z);B)$, $z\in {\rm range}(f)$, is upper-semicontinuous. In particular, for every $\varepsilon>0$, set $S_\varepsilon :=\{z\in {\rm range}(f)\, :\, C_f(z)\ge\varepsilon\}$ is compact, and each  $f^{-1}(z)\subset K$, $z\in S_\varepsilon$, is non-pluripolar. Suppose that $S_\varepsilon$ is non-polar. Then set $S_\varepsilon\times\mathbb C\subset\mathbb C^{N+1}$ is non-pluripolar. According to Proposition \ref{clp}, there exists a non-identically $-\infty$ plurisubharmonic function $u$ on $\mathbb C^{N+1}$ such that $u|_{\Gamma_f}=-\infty$. Since each set $f^{-1}(z)\times\{z\}$, $z\in S_\varepsilon$, is non-pluripolar in $\mathbb C^n\times\{z\}$, $u|_{\mathbb C^N\times\{z\}}=-\infty$ for all $z\in S_\varepsilon$. Hence, $u=-\infty$ on non-pluripolar set $S_\varepsilon\times\mathbb C$ and therefore it must be $-\infty$ everywhere, a contradiction showing that $S_\varepsilon$ is polar for each $\varepsilon>0$. We set 
\[
S_f:=\bigcup_{n=1}^\infty S_{\frac 1n}.
\]
Then $S_f\subset\mathbb C$ is a polar set and for each $c\in {\rm range}(f)\setminus S_f$,
${\rm cap}(f^{-1}(c);B)=0$, that is, $f^{-1}(c)$ is pluripolar, as required.
\end{proof}

\end{document}